\documentclass[11pt, a4paper]{article}
\usepackage[utf8]{inputenc}
\usepackage[T1]{fontenc}
\usepackage{amsmath, amssymb, amsthm, amsfonts}
\usepackage{geometry}
\usepackage{bm}
\usepackage{subcaption}
\usepackage{graphicx} 
\usepackage{mathrsfs}

\usepackage{algorithm}
\usepackage{algorithmic}
\usepackage{amsmath, amsthm, amssymb, amsfonts}

\usepackage[numbers,square]{natbib}
\usepackage[utf8]{inputenc}	
\usepackage[T1]{fontenc}	
\usepackage{xcolor}		
\usepackage[colorlinks = true,
            linkcolor = purple,
            urlcolor  = blue,
            citecolor = cyan,
            anchorcolor = black]{hyperref}	
\usepackage{booktabs} 		
\usepackage{nicefrac}		
\usepackage{microtype}		
\usepackage{lineno}		
\usepackage{float}			

\usepackage{lipsum}		

\usepackage{newfloat}
\usepackage{xcolor}
\usepackage{mathtools, amssymb, amsfonts, dsfont, nccmath}
\usepackage{microtype}
\usepackage{subcaption}
\usepackage{algorithmic}
\usepackage{graphicx}
\usepackage{textcomp}
\usepackage{xcolor}
\usepackage{amsthm} 
\usepackage{xcolor}

\newtheorem{lemma}{Lemma}
\newtheorem{theorem}{Theorem}
\theoremstyle{definition}
\newtheorem{assumption}{Assumption}
\newtheorem{definition}{Definition}
\newtheorem{remark}{Remark}
\newtheorem{proposition}{Proposition}

\DeclareFloatingEnvironment[name={Supplementary Figure}]{suppfigure}
\usepackage{sidecap}
\sidecaptionvpos{figure}{c}

\newcommand{\M}{\mathcal{M}}

\newcommand{\diff}{\mathrm{d}}

\title{Linear-Quadratic Mean Field Games with Hybrid Local-Global Interactions on Manifolds\thanks{This work is supported by AFOSR Grant FA9550-23-1-0015 and NSERC Grant RGPIN-2019-05336.}}

\usepackage{eso-pic}
\usepackage{tikz}
\usepackage{xcolor}
\PassOptionsToPackage{colorlinks=true,linkcolor=gray,urlcolor=gray}{hyperref}

\usepackage{titling}
\usepackage{footmisc}
\newcommand{\Author}[2]{
  \textbf{#1}\textsuperscript{#2} %
}

\author{
  \Author{Tao Zhang}{1}
}

\date{%
  \textsuperscript{1}Department of Electrical and Computer Engineering, McGill University
}

\begin{document}

\maketitle

\begin{abstract}
This paper studies linear-quadratic mean field games on compact Riemannian manifolds with a hybrid interaction topology. The network structure is a superposition of a deterministic graph for local geometric connectivity and a stochastic directed graph for non-local interactions. The global graph is constructed via random sampling based on a continuous kernel $K$. The out-degree of each node scales as $\Theta(\log N)$ {or as $\Theta(N)$ to represent a sparse or dense network, respectively}. In the infinite-population limit, the continuum system is governed by a coupled system of forward-backward partial differential equations, where the dynamics of the expected state incorporate the {integral operator corresponding to the non-local sampling}. {The existence of a Nash equilibrium is established for this limit system.} Furthermore, {the approximation error is analyzed} using operator concentration inequalities and analytic semigroup theory. {Non-asymptotic high-probability error bounds between the finite-population empirical state and the continuum limit are derived.} The convergence rates {differ depending on} the two topological regimes. Under the dense regime, the tracking error {exhibits a polynomial decay rate dependent on the manifold dimension and Sobolev regularity}, while under the sparse regime, the error decays at a rate of $\mathcal{O}((\log N)^{-1/2})$.
\end{abstract}
\vspace{0.35cm}




\section{Introduction}

Mean Field Game (MFG) theory \cite{lasry2006jeux, lasry2007mean, huang2006large,huang2007large} studies the asymptotic behavior of large systems of interacting agents. By invoking the assumption of exchangeability, the interaction of an individual with the population is approximated by a deterministic mean field, which decouples the dynamics. This framework is {applied to} financial engineering \cite{firoozi2017anoptimal, firoozi2017ameanfield}, smart grid energy management \cite{kizilkale2019integral}, traffic routing \cite{huang2021dynamic}, epidemic control \cite{cho2020mean}, and multi-agent reinforcement learning \cite{lauriere2022learning, cui2021approximately}.

Standard MFG theory assumes a complete interaction graph where agents are asymptotically indistinguishable. To model network heterogeneity, Graphon Mean Field Games (GMFG) \cite{gao2020linear, caines2021graphon,hu2023graphon,bayraktar2023graphon} use graphons (see \cite{lovasz2012large}) to describe the limit of graph sequences. As the population size $N \to \infty$, the discrete adjacency matrix converges to a bounded compact integral operator. The graphon framework assumes the graph sequence is dense, requiring the average nodal degree to scale as $\Theta(N)$. Developments in this direction include studies on common noise \cite{bayraktar2023graphon,zhang2025perturbation}, nonlinear dynamics \cite{coppini2024nonlinear}, and stationary analysis \cite{tchuendom2022stationary}.

For sparse networks, the Law of Large Numbers does not apply in the standard manner, and the formulation of MFGs on such structures remains less developed. Borgs et al. \cite{borgs2018sparse} introduced the theory of graphexes to study limits of sparse graphs. Interacting diffusion and games on locally tree-like sparse graphs are analyzed via local weak convergence \cite{lacker2022, oliveira2019interacting}. Cui, Fabian, and Koeppl \cite{fabian2024learning, cui2021approximately} applied local weak limits to analyze reinforcement learning in MFGs on sparse networks. Recently, a general theory in terms of graphexon \cite{caines2022embedded,caines2024sparse} was developed to define measure limits for both dense and sparse networks. Extending this formulation, the Laplexion concept characterizes local diffusion processes on network limits \cite{caines2025ifac_mean,caines2025cdc_mean,zhang2026IFAC}. {However, these existing limits mainly address pure topological limits or single-scale local diffusion, which are insufficient to capture the coupled multi-scale effects when long-range random jumps and local geometric diffusion coexist.}

In practice, many complex networks exhibit both high local clustering and long-range connections \cite{watts1998collective, barranca2015low}. This hybrid network structure is motivated by large-scale multi-agent systems where physical local diffusion and stochastic global transmission coexist. In spatial epidemic control games, for instance, the deterministic graph Laplacian captures the local virus diffusion caused by physical contact among adjacent communities, while the stochastic global graph models long-range disease transmission induced by intercity travel. Agents optimize their intervention efforts by anticipating both local geometric diffusion and non-local topological jumps. Similarly, in smart grid systems, generators are coupled locally via physical transmission lines, while simultaneously interacting over long distances through random bilateral trading contracts. 

To capture this topology, we model the underlying network on a compact Riemannian manifold as a superposition of a deterministic local graph and a stochastic global graph. The local component models geometric interactions through nearest-neighbor coupling, whose continuum limit corresponds to a geometric realization of the Laplexion structure. As the local connectivity radius $r_N \to 0$, the associated spatial graph Laplacian converges to the Laplace-Beltrami operator $\frac{1}{2}\Delta_g$ \cite{rosenberg1997laplacian,zhang2026IFAC}. The global component models global interactions formed by random sampling based on a {continuous kernel $K$}. The out-degree scales as $\Theta(\log N)$ in the sparse regime and $\Theta(N)$ in the dense regime.

This paper studies Linear-Quadratic-Gaussian (LQG) games on this hybrid graph structure. We derive the continuum limit characterized by a system of Forward-Backward Partial Differential Equations (FBPDEs). The expected state $\mathbf{m}_t$ and the adjoint variable $\Theta_t$ are governed by the operator
$ \mathcal{A} \triangleq a\mathbb{I} + c\mathbf{G} + \frac{\gamma}{2} \Delta_g$
where the dynamics of $\mathbf{m}_t$ incorporate the deterministic integral operator $\mathbf{G}$ characterizing the global sampling process. We prove the existence of the Nash solution for this limit system.

Furthermore, we evaluate the $L^2(\mathcal{M})$ approximation error between the finite-population empirical state $\bar{\mathbf{x}}_t^N$ and the limit state $\mathbf{m}_t$. {Specifically, for any given confidence level $\delta \in (0,1)$, the tracking error is bounded by a deterministic decay rate with probability at least $1-\delta$.} The convergence rates differ {depending on} the topological regimes. For the sparse regime with an out-degree scaling of $\Theta(\log N)$, the supremum error over time $t \in [0,T]$ is bounded by $\mathcal{O}((\log N)^{-1/2})$. For the dense regime scaling as $\Theta(N)$, the error {decays at a polynomial rate that} depends on the manifold dimension $p$ and the Sobolev regularity $s$. These bounds are established using concentration inequalities for random variables in the Hilbert-Schmidt operator space.

The paper is organized as follows. Section \ref{sec:random_graph} formulates the random graph sequence. Section \ref{sec:LQGHGMFG} derives the limit FBPDE system and establishes the well-posedness of the Nash solution. Section \ref{sec:error_analysis} provides the probability error bounds for the sparse and dense network limits. Simulation results are presented in Section \ref{sec:simulation}. Section \ref{sec:conclusion} concludes the paper.
\section{Preliminaries}\label{sec:preliminaries}

Let $(\mathcal{M}, g)$ be a $p$-dimensional compact connected Riemannian manifold without boundary. Let $d_g$ denote the geodesic distance and $\diff V_g$ denote the Riemannian volume measure. We denote the volume of the manifold by $\mathrm{Vol}(\mathcal{M})\triangleq \int_{\mathcal{M}} \diff V_g$, and for simplicity, we assume $\text{Vol}(\mathcal{M}) = 1$ throughout the paper.

Let $L^2(\mathcal{M})$ be the Hilbert space of square-integrable real-valued functions on $\mathcal{M}$ with respect to $\diff V_g$, equipped with the norm
\begin{equation}
\|f\|_{L^2} \triangleq \left( \int_{\mathcal{M}} |f(\alpha)|^2 \diff V_g(\alpha) \right)^{1/2}.
\end{equation}

Let $\Delta_g$ denote the Laplace-Beltrami operator on $\mathcal{M}$. Let $0 = \lambda_0 < \lambda_1 \le \dots$ be the eigenvalues of $-\Delta_g$ and let $\{\phi_k\}_{k=0}^\infty$ be the associated complete orthonormal basis of eigenfunctions in $L^2(\mathcal{M})$. For any real $s \ge 0$, the fractional Sobolev space $H^s(\mathcal{M})$ is defined as the space of functions $f \in L^2(\mathcal{M})$ satisfying
\begin{equation}
\|f\|_{H^s} \triangleq \left( \sum_{k=0}^\infty (1+\lambda_k)^s |\langle f, \phi_k \rangle_{L^2}|^2 \right)^{1/2} < \infty.
\end{equation}

Let $C^k(\mathcal{M})$ denote the space of $k$-times continuously differentiable functions on $\mathcal{M}$. Let $C^{0,\alpha}(\mathcal{M})$ denote the space of Hölder continuous functions with exponent $\alpha \in (0,1]$. 

For any two Banach spaces $X$ and $Y$, the notation $X \hookrightarrow Y$ indicates that $X$ is continuously embedded into $Y$. This implies there exists a constant $C > 0$ such that $\|f\|_Y \le C \|f\|_X$ for all $f \in X$.

Let $\mathcal{L}(X, Y)$ denote the Banach space of bounded linear operators from $X$ to $Y$, equipped with the induced operator norm
\begin{equation}
\|A\|_{X \to Y} \triangleq \sup_{f \in X, f \neq 0} \frac{\|Af\|_Y}{\|f\|_X}.
\end{equation}
When $X = Y$, the notation is abbreviated to $\mathcal{L}(X)$ and the corresponding norm is denoted by $\|\cdot\|_{\mathcal{L}(X)}$.

For any finite set $S$, $|S|$ denotes its cardinality. For a measurable subset $A \subseteq \mathcal{M}$, $\mathrm{Vol}(A)$ denotes its Riemannian volume $\int_A \diff V_g$. The indicator function of a set $A$ is denoted by $\mathbf{1}_A$. Define $\mathbf{1}_{\mathcal{M}}$ as the constant $1$-valued function on $\M$.

\section{Random Graph Model}\label{sec:random_graph}
The underlying interaction network is formalized as a hybrid graph framework defined on a single geometric state space. As illustrated in Figure \ref{fig:hybrid_graph}, this hybrid topology is modeled as the superposition of two complementary structures: a deterministic graph core that dictates short-range geometric diffusion, and a stochastic graph that governs long-range non-local information jumps. This architecture ensures that local and global interactions are coupled within the same spatial domain rather than operating in isolated network layers.
\begin{figure}
    \centering
    \includegraphics[width=0.6\linewidth]{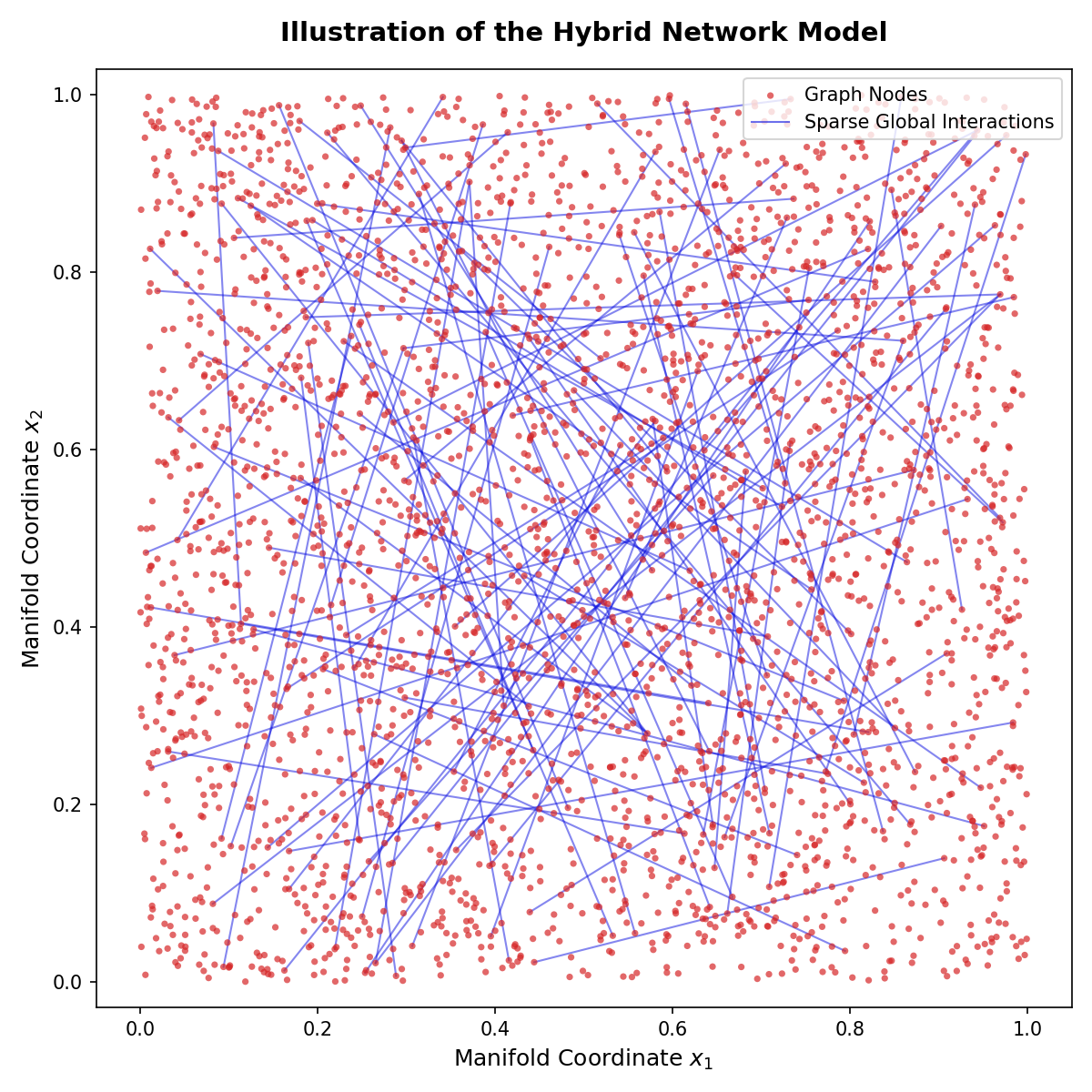}
    \caption{Hybrid Network Model. The red dots represent the agent clusters densely embedded in the underlying manifold $\mathcal{M}$, providing local geometric connectivity. The blue lines represent the sparse, random, non-local interactions that couple distant agents.}
    \label{fig:hybrid_graph}
\end{figure}

We work on a complete probability space constructed as the product
\begin{equation}
(\Omega,\mathcal F,\mathbb P) = (\Omega_w\times\Omega_z, \mathcal{F}_w\otimes\mathcal{F}_z, \mathbb P_w\otimes\mathbb P_z),
\end{equation}
where the two components represent independent sources of randomness.

The first component $(\Omega_w,\mathcal{F}_w,\mathbb P_w)$ supports a family of independent standard Wiener processes $\{w_t^i\}_{i\in\mathbb N}$, which model the idiosyncratic diffusion noise affecting each agent. The associated filtration $\{\mathcal F_w(t)\}_{t\ge0}$ is assumed to satisfy the usual conditions.

The second component $(\Omega_z,\mathcal{F}_z,\mathbb P_z)$ accounts for the random sampling of global neighbors in the sparse interaction structure. All sampling variables are assumed to be independent of the Brownian motions.

Throughout the paper, for any integrable random variable or process $X$, the notation $\mathbb E_z[X]$ denotes the conditional expectation with respect to the sampling randomness only, namely
\begin{equation}
\mathbb E_z[X](\omega_w) = \int_{\Omega_z} X(\omega_w,\omega_z) d\mathbb P_z(\omega_z).
\end{equation}

\subsection{Local Graph Interaction}

The local graph represents nearest-neighbor geometric interactions. To establish a continuum limit for these short-range discrete diffusions, the network nodes must be embedded into a continuous geometric space. 

Let $(\mathcal{M}, g)$ be a compact, connected Riemannian manifold of dimension $p$, equipped with its geodesic distance $d_g$ and Riemannian volume measure $\diff V_g$.

\begin{definition}\label{def:embedded_graph_sequence_simple}
Let $\mathcal{G}^N=(V^{N},E^{N})$ be a finite graph without self-loops, with node set $V^{N} \triangleq \{1,...,{N}\}$ and undirected edge set $E^{N}$. A sequence $\{(\mathcal{G}^N,\varphi^{N})\}_{{N}=1}^\infty$ forms an embedded graph sequence if each graph $\mathcal{G}^N$ is embedded into $\mathcal{M}$ via a node mapping $\varphi^{N} :V^{N} \rightarrow \mathcal{M}$, assigning a spatial location $\alpha_i^{N} = \varphi^{N}(i) \in \mathcal{M}$ to every node $i \in V^{N}$.
\end{definition}
Let $\{\mathcal{P}_N\}_{N\ge 1}$ be a sequence of partitions of the manifold $\mathcal{M}$ into $N$ disjoint measurable sets $\{A_1, \dots, A_N\}$. Let $\delta_N$ denote the maximum diameter of the manifold partition $\{A_i\}_{i=1}^N$. Assume $\text{Vol}(A_i) = 1/N$ for each $i=1,\cdots,N$. Suppose that for any network size $N$ and any discrete node index $i \in V^N$, $\varphi^N$ satisfies $\alpha_i^N = \varphi^N(i) \in A_i$.

To construct the local discrete interaction operators, introduce the connectivity radius $r_N > 0$ satisfying $r_N \to 0$ as $N \to \infty$. The local neighbor set $\mathcal{N}_i$ of node $i$ is defined by
\begin{equation}
    \mathcal{N}_i \triangleq \left\{ j \in V^N \setminus \{i\} : d_g(\alpha_i^N, \alpha_j^N) \le r_N \right\}.
\end{equation}
For each node $i$ with location $\alpha_i^N\in \M$, we choose an orthonormal frame $E_i=(e_1,\dots,e_q)$ of $T_{\alpha_i^N}\mathcal{M}$. 
We identify $T_{\alpha_i^N}\mathcal{M}\simeq \mathbb{R}^p$ via $E_i$.
For each neighbor $j\in\mathcal{N}_i$ we define the (geodesic) displacement vector in the tangent space $T_{\alpha_i^{N}}\mathcal{M}$ by the exponential map $ v_{ij}\;\triangleq\;\exp_{\alpha_i^N}^{-1}(\alpha_j^N)\ \in\ T_{\alpha_i^N}\mathcal{M}$
and write it in coordinates under $E_i$ as $v_{ij}\;=\;\sum_{k=1}^p v_{ij}^k\,e_k$, i.e., $v_{ij}\equiv \begin{bmatrix} v_{ij}^1 & \cdots & v_{ij}^p \end{bmatrix}^{\!\top}\in\mathbb{R}^p$.

Let the symmetric edge weights be $w_{ij}^{(N)} = C_w r_N^{-2}$ for $j \in \mathcal{N}_i$ and $w_{ij}^{(N)} = 0$ otherwise, where $C_w$ is a dimension-dependent scaling constant. Let $\mathbb{T}_N \in \mathbb{R}^{N \times N}$ denote the discrete graph Laplacian matrix associated with $\mathcal{G}^N$. Its algebraic action on a state vector ${v} \in \mathbb{R}^N$ is given by
\begin{equation}
    [\mathbb{T}_N {v}]_i \triangleq \frac{1}{|\mathcal{N}_i|}\sum_{j\in\mathcal{N}_i} w_{ij}^{(N)} (v_j - v_i), \qquad i \in V^N.
\end{equation}

 The graph Laplacian operator $\mathbf{L}_N : L^2(\mathcal{M}) \to L^2(\mathcal{M})$ is defined by the local integral averages
\begin{equation}\label{eq:def_L_N}
    (\mathbf{L}_N \psi)(\alpha) \triangleq \sum_{i=1}^N \mathbf{1}_{A_i}(\alpha) \frac{1}{|\mathcal{N}_i|}\sum_{j\in\mathcal{N}_i} w_{ij}^{(N)} \big(\psi_{A_j} - \psi_{A_i}\big)
\end{equation}
 for any function $\psi \in L^2(\mathcal{M})$ with 
 \begin{equation}
\psi_{A_j} \triangleq \frac{1}{\text{Vol}(A_j)}\int_{A_j}\psi dV_g.
\end{equation}

The geometric consistency of the embedded graph sequence is characterized by the local empirical moments.

\begin{assumption}[\cite{zhang2026IFAC}]\label{as:uniform_isotropic_graph}
Assume $\delta_N = \mathcal{O}(N^{-1/p})$ and $r_N = \mathcal{O}(\delta_N)$. Suppose there exists a constant $C_u > 0$ such that for all nodes $i \in V^N$,
\begin{align}
    \left\| \frac{1}{|\mathcal{N}_i|} \sum_{j\in\mathcal{N}_i} w_{ij}^{(N)} v_{ij} \right\| &\le C_u r_N, \\
    \left\| \frac{1}{|\mathcal{N}_i|} \sum_{j\in\mathcal{N}_i} w^{(N)}_{ij} v_{ij} v_{ij}^{\top} - I_p \right\| &\le C_u r_N,
\end{align}
where $I_p$ is the $p$-dimensional identity matrix. 
\end{assumption}

\begin{lemma}[\cite{zhang2026IFAC}] \label{prop:laplace_convergence}
    Under Assumption \ref{as:uniform_isotropic_graph}, for any test function $\psi \in C^3(\mathcal{M})$, the graph Laplacian converges to the Laplace-Beltrami operator
    \begin{equation}\label{eq:laplace_limit}
         \lim_{{N} \to \infty} \mathbf{L}_N \psi = \frac{1}{2} \Delta_g \psi.
    \end{equation}
\end{lemma}
\subsection{Global Directed Graph Interaction}\label{sec:global_graph}

In parallel with the local geometric diffusion, the stochastic global graph introduces non-local interactions that connect distant regions of the same state space $\mathcal{M}$. Let all random variables and processes governing this global graph construction be defined on a common probability space $(\Omega_z, \mathcal{F}_z, \mathbb{P}_z)$. The sparse connection probability of any given node to the global network is characterized by a continuous kernel $g$ parameterized by its spatial label. 

Let $\{\nu_\alpha\}_{\alpha \in \mathcal{M}}$ be a family of probability measures on the manifold $\mathcal{M}$ such that $\nu_\alpha \in \mathcal{P}(\mathcal{M})$.

\begin{assumption}
    Assume the map $\alpha \mapsto \nu_\alpha$ is measurable with respect to the Borel $\sigma$-algebra $\mathcal{B}(\mathcal{M})$. For almost every $\alpha \in \mathcal{M}$ with respect to the volume measure, the measure $\nu_\alpha$ is absolutely continuous with respect to $\diff V_g$.
\end{assumption}

\begin{definition}
    The kernel density function associated with $\{\nu_\alpha\}_{\alpha \in \mathcal{M}}$ is defined for almost every $\alpha,\beta \in \mathcal{M}$ by
    \begin{equation}
K(\alpha, \beta) \triangleq \frac{\diff \nu_\alpha}{\diff V_g}(\beta).
\end{equation}
\end{definition}

Define the integral operator $\mathbf{G}: L^2(\mathcal{M}) \to L^2(\mathcal{M})$ acting on any test function $f \in L^2(\mathcal{M})$ at $\alpha \in \mathcal{M}$ by
\begin{equation}
    (\mathbf{G}f)(\alpha) \triangleq \int_{\mathcal{M}} f(\beta) \nu_\alpha(\diff \beta) = \int_{\mathcal{M}} f(\beta) K(\alpha,\beta)\diff V_g(\beta).
\end{equation}

\begin{assumption}
\label{asm:kernel_lipschitz}
Assume there exists a constant $C_g > 0$ such that for almost every $\beta \in \mathcal{M}$, the map $\alpha \mapsto K(\alpha, \beta)$ belongs to $H^2(\mathcal{M})$ and satisfies the uniform norm bound
\begin{equation}
\sup_{\beta \in \mathcal{M}} \|K(\cdot, \beta)\|_{H^2(\mathcal{M})} \le C_g.
\end{equation}
\end{assumption}
From Assumption \ref{asm:kernel_lipschitz}, one can obtain that the operator $\mathbf{G}$ is bounded.

Given the embedded graph sequence $\{(\mathcal{G}^N, \varphi^{N})\}_{{N} \ge 1}$ and the manifold partition $\{A_k\}_{k=1}^N$, the continuous operator is discretized to form the finite-dimensional interaction probabilities.

\begin{definition}\label{def:canonical_discretization}
For any nodes $q, k \in V^N$ and let $\alpha_q^N \in A_q$ be the center point of $A_q$, the matrix $G^{[N]}$ is defined component-wise by
\begin{equation}
    G^{[N]}_{qk} \triangleq \int_{A_k} K(\alpha_q^N, \beta) \diff V_g(\beta).
\end{equation}
\end{definition}

Since $\nu_{\alpha_q^N}$ is a probability measure on $\mathcal{M}$ and the sequence of partitions consists of disjoint measurable sets covering the entire manifold, the sum of probabilities over all target cells equals one. Consequently, the sequence of matrices $G^{[{N}]} \in [0,1]^{{N} \times {N}}$ is row-stochastic.

Given a deterministic degree vector $\mathbf{d}^{[{N}]} = [d_1, d_2, \ldots, d_{N}]^\top$ with $d_q \in \mathbb{{N}}^+$, the random neighbor set is constructed for each node $q \in V^N$. Let $\{L_i^{[{N}],q}\}_{i=1}^{d_q}$ be i.i.d. random variables taking values in $V^N$ according to $\mathbb{P}_z(L_i^{[{N}],q} = k) = G^{[{N}]}_{qk}$. The random global neighbor set of node $q$ is
\begin{equation}
    \mathcal{{N}}_q^{(1)} \triangleq \{L_1^{[{N}],q}, L_2^{[{N}],q}, \ldots, L_{d_q}^{[{N}],q}\}.
\end{equation}

\begin{assumption}\label{asm:general_degree}
    Given two variables $d_{min},d_{max}$ satisfying $d_{\max}/d_{\min} \le c_0$ and $d_{\min} \ge c_1 \log N$ for some constants $c_0>0,c_1\ge 1$. Assume that the out-degree sequence satisfies $1 \le d_{\min} \le d_q \le d_{\max}$ for all nodes $q \in V^N$.
\end{assumption}

Two asymptotic regimes are considered for the out-degree sequence.

The first regime is the sparse network, where the degree sequence scales logarithmically with the network size, i.e., there exist constants $c_2 \ge c_1 > 0$ such that $c_1 \log N \le d_{min} \le c_2 \log N$. The total number of edges scales as $\mathcal{O}(N \log N)$, yielding an asymptotic edge density of $\mathcal{O}(\frac{\log N}{N})$.

The second regime is the dense network, where the out-degree sequence scales linearly with the network size, satisfying $d_{min} = \Theta(N)$.

\section{LQG Hybrid Graph Mean Field Games}\label{sec:LQGHGMFG}

To maintain clarity throughout the analysis on the manifold, the following notational conventions are adopted:
\begin{itemize}
    \item \textbf{Continuum Objects}: Variables in boldface (e.g., $\mathbf{x}, \mathbf{m}, \mathbf{G}$) represent infinite dimensional objects defined on the continuum manifold $\mathcal{M}$.
    
    \item \textbf{Averaged States}: The overbar notation ($\bar{\cdot}$) is reserved for cluster-wise averaged states. For instance, $\bar{x}_t^{[N],q}$ denotes the empirical average state of the finite-node cluster $q$.
\end{itemize}
\subsection{Model Setup}

Consider a hybrid network comprising core-interaction and global-interaction structures where each node represents a cluster of homogeneous agents. Different agents $\mathcal{A}_i,\mathcal{A}_j$ of the same cluster share the same dynamics parameters. Let $\mathcal{C}_q, q \in \{1,\ldots,N\}$ denote the set of agents in cluster $q$, with total population $L=\sum_{q=1}^N{|\mathcal{C}_q|}$. For any agent $\mathcal{A}_i \in \mathcal{C}_q$, denote by $x_t^i, u_t^i \in \mathbb{R}$ its state and control input at time instant $t \in [0,T]$, respectively. While the focus is primarily on one-dimensional state spaces for analytical tractability, the framework naturally extends to multi-dimensional cases.

 To model sparse, random interactions across the network, for a node $\alpha_q \in \mathcal{M}$, assume that it connects to $d_q$ external neighbors. These neighbors are sampled independently and identically from a probability measure $\nu_{\alpha_q}$ on $\mathcal{M}$.
The empirical mean field from these peripheral neighbors is defined by
\begin{equation}
z_t^{[N], q} \triangleq \frac{1}{d_q} \sum_{j=1}^{d_q} \bar{x}_t^{[N], L_j^{[N], q}}= (M^{[N]} \bar{x}_t^{[N]})_q
\end{equation}
where $|\mathcal{N}_q^{(1)}| = d_q$, $\bar x_t^{[N],q}\triangleq\frac{1}{|\mathcal{C}_q|}\sum_{i=1}^{|\mathcal{C}_q|} x_t^i$ denotes the average state of cluster $\mathcal{C}_q$, and $\bar{x}_t^{[N]} \triangleq [\bar{x}_t^{[N],1}, \dots, \bar{x}_t^{[N],N}]^\top$.

\begin{remark}\label{rem:lln_limitation}
The convergence of the empirical mean field $z_t^{[N],q}$ to a deterministic continuum limit cannot be established by a direct application of the Law of Large Numbers. The state variables $\bar{x}_t^{[N],l}$ within the random neighbor set are not independent and identically distributed; their dynamics are continuously coupled through the realization of the random graph topology and the local discrete Laplacian. 
\end{remark}

Based on the geometric graph connectivity established in Section \ref{sec:random_graph}, the local interaction is governed by the discrete graph Laplacian. Denote the deterministic initial state of $\mathcal{A}_i$ by $x^i_0$. The dynamics of any agent $\mathcal{A}_i\in \mathcal{C}_q$ are governed by
\begin{equation}
\label{eq:individual_dynamics}
 \diff x_t^i = (a x_t^i + bu_t^i + cz_t^{[N],q}+\gamma (\mathbb{T}_N \bar{x}_t^{[N]})_q)\diff t+\sigma \diff w_t^i, \quad |x^i_0|<\infty
\end{equation}
where $a, b, c\in \mathbb{R},\gamma>0$ are constants and $ \sigma \in \mathbb{R}$. The initial condition is assumed to be independent of the standard i.i.d. Wiener processes $\{w_t^i\}_{i=1,\ldots,N}$. 

The individual cost function of agent $\mathcal{A}_i$ is formulated as
\begin{equation}\label{eq:cost_finite}
 J^{i}=\mathbb{E}\left [\int _{0}^{T}\left(q_0\left(x_t^{i}-\phi_t^{[N],q}\right)^2+r(u_t^{i})^2\right)\diff t  +q_T(x_T^{i}-\phi_T^{[N],q})^2\right]
\end{equation}
where $q_0 \ge 0,r >0$, $\phi_t^{[N],q} \triangleq Hz_t^{[N],q}+\eta$ for any $t \in [0,T]$, and $q_T\ge0,H \in \mathbb{R},\eta \in \mathbb{R}$ are known constants.

\begin{definition}
 A control tuple $(u^1,...,u^L)$ is said to be a Nash equilibrium if any unilateral deviation from $u^i$ to any other $\hat u^i$ yields a higher cost, that is
 \begin{equation}
J^i(u^i,u^{-i}) \le J^i(\hat u^i,u^{-i}), \forall i=1,2,...,L
\end{equation}
 where $u^{-i} \triangleq(u^1,...,u^{i-1},u^{i+1},...,u^L)$.
\end{definition}

Consider the limiting case where the local population of each cluster approaches infinity, i.e., $|\mathcal{C}_q| \rightarrow \infty$ for all $q \in \{1,2,\ldots,N\}$. Let $\mu^{[N],q}_t$ denote the empirical measure of the states in cluster $\mathcal{C}_q$, then the average state of cluster $\mathcal{C}_q$ becomes
\begin{equation}
\bar{x}_t^{[N],q}=\lim _{|\mathcal{C}_q| \rightarrow \infty}\frac{1}{|\mathcal{C}_q|}\sum_{i=1}^{|\mathcal{C}_q|} {x}_t^i=\int_\mathbb{R} x\mu_t^{[N],q}(\diff x)
\end{equation}

Let $\mathcal{A}_\alpha \in \mathcal{C}_q$ denote a generic individual agent in the infinite-population cluster $q$. Its state and control are denoted by $x_t^\alpha$ and $u_t^\alpha$. Conditional on a fixed realization of the random graph sampling $\omega \in \Omega_z$, the states $\bar{x}_t^{[N],q}$ and the associated mean-field $z_t^{[N],q}$ become deterministic. The dynamics of this generic agent $\alpha$ are governed by
\begin{equation}
    \diff x_t^\alpha = \left( a x_t^\alpha + b u_t^\alpha + c z_t^{[N],q} + \gamma (\mathbb{T}_N \bar{x}_t^{[N]})_q \right) \diff t + \sigma \diff w_t^\alpha, \quad |x^\alpha_0|<\infty
\end{equation}
where $w_t^\alpha$ is a standard Wiener process. The corresponding individual cost function is
\begin{equation}
    J^\alpha = \mathbb{E} \left[ \int_0^T \left( q_0 (x_t^\alpha - \phi_t^{[N],q})^2 + r (u_t^\alpha)^2 \right) \diff t + q_T (x_T^\alpha - \phi_T^{[N],q})^2 \right].
\end{equation}

Let ${z}_t^{[N]} \triangleq [z_t^{[N],1}, \dots, z_t^{[N],N}]^\top$. Let the continuum initial state be a sufficiently smooth function $\bar{\mathbf{x}}_0 \in H^2(\mathcal{M})$ and assume that $\bar{x}^{[N],q}_0=\bar{\mathbf{x}}_0(\alpha)$ for the center point $\alpha\in A_q$.

To establish the mean-field limit, we first characterize the optimal decentralized control for the finite-dimensional system given a fixed graph realization.

\begin{lemma}
Given a fixed realization of the random graph topology $\omega \in \Omega_z$. Suppose there exists a pair of processes $(\bar{x}_t^{[N]}, {\theta}_t^{[N]})$ satisfying the following finite-dimensional Forward-Backward ODE system
\begin{equation}
\left\{
\begin{aligned}
    \frac{\diff \bar{x}^{[N]}_t}{\diff t} &= a \bar{x}^{[N]}_t - \frac{b^2}{r}{\theta}^{[N]}_t + c M^{[N]}(\omega)\bar{x}^{[N]}_t + \gamma \mathbb{T}_N \bar{x}^{[N]}_t, \\
    -\frac{\diff {\theta}^{[N]}_t}{\diff t} &= a {\theta}^{[N]}_t + q_0 \bar{x}^{[N]}_t - q_0 {\phi}_t^{[N]}, \\
    {\theta}_T^{[N]} &= q_T \bar{x}_T^{[N]} - q_T {\phi}_T^{[N]},
\end{aligned}
\right.
\end{equation}
where $\Pi_t$ solves the following Riccati equation
\begin{equation}\label{eq:riccati}
    -\dot{\Pi}_t = 2a\Pi_t - \frac{b^2}{r}\Pi_t^2 + q_0,\qquad \Pi_T = q_T,
\end{equation}
and ${\phi}_t^{[N]} \triangleq H {z}_t^{[N]} + \eta$. Then the optimal decentralized control law for the generic agent $\mathcal{A}_\alpha \in \mathcal{C}_q$ is given by
\begin{equation}
    u_t^\alpha = -\frac{b}{r} \left( \Pi_t (x_t^\alpha - \bar{x}_t^{[N],q}) + \theta_t^{[N],q} \right).
\end{equation}
\end{lemma}

\begin{proof}
The proof relies on the verification argument for linear-quadratic mean field games. For detailed derivations on similar network models, see \cite{gao2021lqg}.
\end{proof}

Under this optimal control policy, taking the expectation over the idiosyncratic noise, the closed-loop dynamics of the cluster average $\bar{x}_t^{[N],q}$ are deterministic and governed by
\begin{equation}\label{eq:cluster_dyn_finite_N}
    \frac{\diff {\bar{x}}^{[N],q}_t}{\diff t} = a \bar{x}^{[N],q}_t - \frac{b^2}{r} \theta_t^{[N],q} + c z^{[N],q}_t + \gamma (\mathbb{T}_N \bar{x}_t^{[N]})_q.
\end{equation}

Recall that $\{\mathcal{P}_N\}_{N\ge 1}$ is a sequence of partitions of the manifold $\mathcal{M}$ into $N$ disjoint measurable sets $\{A_1, \dots, A_N\}$. 

For any $q=1,\cdots,N$, let $\alpha\in A_q$ be any point of $A_q$. The discretized random sampling operator $\Phi_N^\omega \in \mathcal{L}(L^2(\mathcal{M}))$ is defined as a piecewise constant operator based on local cell averages
\begin{equation} \label{eq:def_Phi_omega_N}
    [\Phi_N^\omega f](\alpha) \triangleq \sum_{q=1}^N \mathbf{1}_{A_q}(\alpha) \frac{1}{d_q} \sum_{i \in \mathcal{N}_q^{(1)}} \frac{1}{\mathrm{Vol}(A_i)} \int_{A_i} f(\beta) \diff V_g(\beta).
\end{equation}
\begin{lemma} \label{lem:L2_norm_bound_operator}
    Let Assumption \ref{asm:general_degree} hold. For any confidence level $\delta \in (0,1)$, there exists a constant $C_{\max}(\delta) > 0$ independent of $N$ such that
    \begin{equation}
\|\Phi_N^\omega\|_{\mathcal{L}(L^2)}  \le C_{\max}(\delta)
\end{equation}
    with probability at least $1-\delta/2$.
\end{lemma}

\begin{proof}
    For any test function $f \in L^2(\mathcal{M})$, define its local cell average vector $\mathbf{f} \in \mathbb{R}^N$ with components $\mathbf{f}_k = \frac{1}{\mathrm{Vol}(A_k)} \int_{A_k} f(\beta) \diff V_g(\beta)$. Applying Jensen's inequality guarantees that
    \begin{equation}
\|\mathbf{f}\|_2^2 = \sum_{k=1}^N \mathbf{f}_k^2 \le \sum_{k=1}^N N \int_{A_k} f(\beta)^2 \diff V_g(\beta) = N \|f\|_{L^2(\mathcal{M})}^2.
\end{equation}
    By the definition of $\Phi_N^\omega$, its $L^2$-norm is determined by $M^{[N]}$
    \begin{equation}
\|\Phi_N^\omega f\|_{L^2(\mathcal{M})}^2 = \sum_{q=1}^N \int_{A_q} \left( \sum_{k=1}^N M^{[N]}_{qk} \mathbf{f}_k \right)^2 \diff V_g(\alpha) = \frac{1}{N} \|M^{[N]} \mathbf{f}\|_2^2.
\end{equation}
    Consequently, 
    \begin{equation}
\|\Phi_N^\omega\|_{\mathcal{L}(L^2)} \le \|M^{[N]}\|_2\le \sqrt{\|M^{[N]}\|_1 \|M^{[N]}\|_\infty}
\end{equation}
where $\|M^{[N]}\|_\infty$ is the maximum absolute row sum and $\|M^{[N]}\|_1$ is the column sum.
    Because $M^{[N]}$ is a row-stochastic matrix by construction, $\|M^{[N]}\|_\infty = 1$. It remains to bound $\|M^{[N]}\|_1 = \max_{1 \le i \le N} \mathcal{C}_i$, where the column sum $\mathcal{C}_i$ for cell $i$ is formulated by
    \begin{equation}
\mathcal{C}_i = \sum_{q=1}^N M^{[N]}_{qi} = \sum_{q=1}^N \frac{1}{d_q} \sum_{j=1}^{d_q} \mathbf{1}_{\{L_j^{[N],q} = i\}}.
\end{equation}
    Let $K_i = \sum_{q=1}^N \sum_{j=1}^{d_q} \mathbf{1}_{\{L_j^{[N],q} = i\}}$ for each cell $i$. Then $\mathcal{C}_i \le K_i / d_{\min}$ and
    \begin{equation}
\mu_i = \mathbb{E}_z[K_i] = \sum_{q=1}^N d_q G^{[N]}_{qi} \le d_{\max} \mu_{\max}
\end{equation}
    where $\mu_{\max} = \|g\|_\infty $.
    Applying the Bernstein's inequality to $K_i$, for any $t > 0$
    \begin{equation}
\mathbb{P}_z ( K_i-\mu_i \ge t ) \le \exp\left( - \frac{t^2}{2\mu_i + 2t/3} \right).
\end{equation}
    To ensure the global failure probability over all $N$ columns does not exceed $\delta/2$, we equate the right-hand side to $\delta / (2N)$ and apply the inequality $\sqrt{A+B} \le \sqrt{A} + \sqrt{B}$:
    \begin{equation}
t \le \frac{2}{3}\ln(2N/\delta) + \sqrt{2\mu_i \ln(2N/\delta)}.
\end{equation}
    Thus, with probability at least $1 - \delta/(2N)$, the individual degree satisfies $K_i \le \mu_i + t \le d_{\max} \mu_{\max}+t$. Therefore, with probability at least $1-\delta/2$,
    \begin{equation}
\max_{1 \le i \le N} K_i \le d_{\max} \mu_{\max} + \frac{2}{3}\ln(2N/\delta) + \sqrt{2 d_{\max} \mu_{\max} \ln(2N/\delta)}.
\end{equation}
    Recall that $\max_i \mathcal{C}_i \le (\max_i K_i) / d_{\min}$. Dividing the above inequality by $d_{\min}$,
    \begin{equation}
\max_{1 \le i \le N} \mathcal{C}_i \le \frac{d_{\max}}{d_{\min}} \mu_{\max} + \frac{2}{3}\frac{\ln(2N/\delta)}{d_{\min}} + \sqrt{2 \frac{d_{\max}}{d_{\min}} \mu_{\max} \frac{\ln(2N/\delta)}{d_{\min}}}.
\end{equation}
    Under Assumption \ref{asm:general_degree}, 
    \begin{equation}
\frac{\ln(2N/\delta)}{d_{\min}} = \frac{\ln N + \ln(2/\delta)}{d_{\min}} \le \frac{1}{c_1} + \frac{\ln(2/\delta)}{c_1 \ln 2} \triangleq B(\delta).
\end{equation}
    Applying the substitution $d_{\max}/d_{\min} \le c_0$ and $B(\delta)$ to the components of $\max_i \mathcal{C}_i$, the uniform bound independent of $N$ is established
    \begin{equation}
\max_{1 \le i \le N} \mathcal{C}_i \le c_0 \mu_{\max} + \frac{2}{3}B(\delta) + \sqrt{2 c_0 \mu_{\max} B(\delta)}.
\end{equation}
    Thus, $\|\Phi_N^\omega\|_{\mathcal{L}(L^2)} \le \sqrt{\max_i \mathcal{C}_i} \triangleq C_{\max}(\delta)$ which holds with probability at least $1-\delta/2$, completing the proof.
\end{proof}

\subsection{Continuum Limit Dynamics}
To establish the continuum limit, the finite-dimensional empirical processes must be mapped to the function space on the manifold. Specifically, the continuum mean field state $\mathbf{m}_t \in L^2(\mathcal{M})$ and the continuum adjoint state $\Theta_t \in L^2(\mathcal{M})$ correspond to the spatial limits of the expected discrete cluster average state $\bar{x}_t^{[N]}$ w.r.t. the neighbor sampling process in Section \ref{sec:global_graph} and the discrete adjoint vector $\theta_t^{[N]}$, respectively.

To analyze the well-posedness of the continuum dynamics, define the unbounded linear operator $\mathcal{A}_0$ on $L^2(\mathcal{M})$ by
\begin{equation}
    \mathcal{A}_0 \triangleq a\mathbb{I} + \frac{\gamma}{2} \Delta_g.
\end{equation}
The domain is $D(\mathcal{A}_0) = H^2(\mathcal{M})$ and $\mathcal{A}_0$ generates a analytic $C_0$-semigroup $S_0(t)$ on $L^2(\mathcal{M})$ \cite{pazy2012semigroups}. Since $\frac{\gamma}{2}\Delta_g$ generates a contraction semigroup, $S_0(t)$ satisfies the bound $\|S_0(t)\|_{\mathcal{L}(H^s)} \le e^{at}$ for any $t \ge 0$ and any real $s \ge 0$. 

Further define the unbounded operator 
\begin{equation}
\mathcal{A} \triangleq a\mathbb{I} + c\mathbf{G} + \frac{\gamma}{2} \Delta_g
\end{equation}
acting as a bounded perturbation of the strongly elliptic Laplace-Beltrami operator and therefore generating an analytic $C_0$-semigroup $S(t)$ on $L^2(\mathcal{M})$.

\begin{theorem} \label{thm:fbpde_local}
Let $\mathcal{K} \triangleq q_0(\mathbb{I} - H\mathbf{G})$ and $\mathcal{K}_T \triangleq q_T(\mathbb{I} - H\mathbf{G})$. Then there exists a maximal time horizon $T^* > 0$ such that for any finite horizon $T \in (0, T^*)$, the coupled Forward-Backward PDE system
\begin{equation}
\label{eq:FBPDE_asymmetric}
\left\{
\begin{aligned}
 \dot{\mathbf{m}}_t &= \mathcal{A} \mathbf{m}_t - \frac{b^2}{r} \Theta_t,  \\
 -\dot{\Theta}_t &= a \Theta_t + \mathcal{K}\mathbf{m}_t - q_0\eta\mathbf{1}_\M,  \\
 \mathbf{m}_0 &= \bar{\mathbf{x}}_0,\\
 \Theta_T &= \mathcal{K}_T \mathbf{m}_T - q_T\eta\mathbf{1}_\M,
\end{aligned}
\right.
\end{equation}
admits a unique solution pair 
\begin{equation}
\mathbf{m}\in C^1([0,T]; L^2(\mathcal{M})) \cap C([0,T]; H^2(\mathcal{M})), \qquad \Theta \in C^1([0,T]; L^2(\mathcal{M})).
\end{equation}
\end{theorem}

\begin{proof}
Let $\mathbf{X}=C([0,T]; L^2(\mathcal{M}))$. For a given state trajectory $\mathbf{m} \in \mathbf{X}$, the unique solution $\Theta \in C([0,T]; L^2(\mathcal{M}))$ is explicitly given by
\begin{equation}
\Theta_t = e^{a(T-t)} (\mathcal{K}_T \mathbf{m}_T - q_T\eta \mathbf{1}) + \int_t^T e^{a(s-t)} (\mathcal{K} \mathbf{m}_s - q_0\eta \mathbf{1}) \diff s.
\end{equation}
This representation defines a continuous mapping $\Gamma_1 \colon \mathbf{X} \to \mathbf{X}$. For any pair $\mathbf{m}^{(1)}, \mathbf{m}^{(2)} \in \mathbf{X}$, the difference $\delta\Theta = \Gamma_1(\mathbf{m}^{(1)}) - \Gamma_1(\mathbf{m}^{(2)})$ satisfies
\begin{equation}
\|\delta\Theta_t\|_{L^2} \le e^{|a|T} \|\mathcal{K}_T\|_{\mathcal{L}(L^2)} \|\delta\mathbf{m}_T\|_{L^2} + \int_t^T e^{|a|(s-t)} \|\mathcal{K}\|_{\mathcal{L}(L^2)} \|\delta\mathbf{m}_s\|_{L^2} \diff s.
\end{equation}
Taking the supremum over $t \in [0,T]$ yields the Lipschitz bound for the backward mapping
\begin{equation}
\|\delta\Theta\|_{\mathbf{X}} \le e^{|a|T} \left( \|\mathcal{K}_T\|_{\mathcal{L}(L^2)} + T \|\mathcal{K}\|_{\mathcal{L}(L^2)} \right) \|\delta\mathbf{m}\|_{\mathbf{X}} \triangleq L_\Theta(T) \|\delta\mathbf{m}\|_{\mathbf{X}}.
\end{equation}

For the forward process, the unbounded operator $\mathcal{A}$ generates an analytic $C_0$-semigroup $S(t)$ on $L^2(\mathcal{M})$. There exist constants $M \ge 1$ and $\omega_0 \in \mathbb{R}$ such that $\|S(t)\|_{\mathcal{L}(L^2)} \le M e^{\omega_0 t}$ for all $t \ge 0$. Let $\omega_0^+ \triangleq \max\{\omega_0, 0\}$.

Given an adjoint trajectory $\Theta \in \mathbf{X}$, the mild solution of the forward equation defines a mapping $\Gamma_2 \colon \mathbf{X} \to \mathbf{X}$ governed by
\begin{equation}
\mathbf{m}_t = S(t)\bar{\mathbf{x}}_0 - \int_0^t S(t-\tau) \frac{b^2}{r} \Theta_\tau \diff \tau.
\end{equation}
For any $\Theta^{(1)}, \Theta^{(2)} \in \mathbf{X}$, the difference $\delta\mathbf{m} = \Gamma_2(\Theta^{(1)}) - \Gamma_2(\Theta^{(2)})$ is bounded by
\begin{equation}
\|\delta\mathbf{m}_t\|_{L^2} \le \int_0^t M e^{\omega_0(t-\tau)} \frac{b^2}{r} \|\delta\Theta_\tau\|_{L^2} \diff \tau.
\end{equation}
Taking the supremum over $[0,T]$ establishes the Lipschitz continuity of the forward mapping
\begin{equation}
\|\delta\mathbf{m}\|_{\mathbf{X}} \le M e^{\omega_0^+ T} \frac{b^2}{r} T \|\delta\Theta\|_{\mathbf{X}} \triangleq L_m(T) \|\delta\Theta\|_{\mathbf{X}}.
\end{equation}

Consider the composite operator $\mathcal{F} \triangleq \Gamma_2 \circ \Gamma_1 \colon \mathbf{X} \to \mathbf{X}$. Combining the Lipschitz estimates yields a total bound
\begin{equation}
\|\mathcal{F}(\mathbf{m}^{(1)}) - \mathcal{F}(\mathbf{m}^{(2)})\|_{\mathbf{X}} \le L(T) \|\mathbf{m}^{(1)} - \mathbf{m}^{(2)}\|_{\mathbf{X}},
\end{equation}
where the total contraction constant is defined by
\begin{equation}
L(T) \triangleq M e^{(\omega_0^+ + |a|)T} \frac{b^2}{r} T \left( \|\mathcal{K}_T\|_{\mathcal{L}(L^2)} + T \|\mathcal{K}\|_{\mathcal{L}(L^2)} \right).
\end{equation}

Since $L(T)$ is monotonically increasing w.r.t. $T$ and satisfies $\lim_{T \to 0^+} L(T) = 0$, there exists a unique horizon $T^* > 0$ such that $L(T^*) = 1$. For any prescribed horizon $T \in (0, T^*)$, the condition $L(T) < 1$ holds. The Banach Fixed-Point Theorem guarantees the existence of a unique fixed point $\mathbf{m} \in \mathbf{X}$. The adjoint state is uniquely determined by $\Theta = \Gamma_1(\mathbf{m}) \in \mathbf{X}$.

Furthermore, note that $\mathbf{m} \in C([0,T]; L^2(\mathcal{M}))$ implies $\Theta \in C^1([0,T]; L^2(\mathcal{M}))$. Consequently, the term $-\frac{b^2}{r} \Theta_t \in C^1([0,T]; L^2(\mathcal{M}))$. Since the initial condition satisfies $\bar{\mathbf{x}}_0  \in H^2(\mathcal{M})$, the maximal regularity property of analytic semigroups guarantees that $\mathbf{m} \in C^1([0,T]; L^2(\mathcal{M})) \cap C([0,T]; H^2(\mathcal{M}))$. This completes the proof.
\end{proof}

\section{Error Analysis in Sparse Limits}\label{sec:error_analysis}

This section quantifies the deviation of the finite-agent network dynamics from the deterministic mean field limit. 

For any $t \in [0,T]$, define the piece-wise constant functions $\bar{\mathbf{x}}_t^N \in L^2(\mathcal{M})$ and $\bar{\boldsymbol{\theta}}_t^N \in L^2(\mathcal{M})$ by 
\begin{equation} \label{eq:piecewise_constant_embedding}
    \bar{\mathbf{x}}_t^N(\alpha) \triangleq \sum_{q=1}^N \mathbf{1}_{A_q}(\alpha) \bar{x}_t^{[N],q}, \qquad \bar{\boldsymbol{\theta}}_t^N(\alpha) \triangleq \sum_{q=1}^N \mathbf{1}_{A_q}(\alpha) \theta_t^{[N],q}, \quad \forall \alpha \in \mathcal{M}.
\end{equation}
By this construction, the spatial limits of these embedded functions correspond to the continuum mean field state $\mathbf{m}_t$ and the continuum adjoint state $\Theta_t$.

Throughout this section, we fix a constant $s$ such that $s \in (p/2, \min\{2, p/2 + 1\})$.

\begin{assumption}\label{asm:manifold_dim}
 Assume that the dimension of $\mathcal{M}$ is $p \le 3$.
\end{assumption}

Define the discrete operator $\mathcal{A}^N_0 \triangleq a\mathbb{I} + \gamma \mathbf{L}_N$ on $L^2(\mathcal{M})$. Let $S_N(t) \triangleq e^{\mathcal{A}^N_0 t}$ denote the analytic $C_0$-semigroup generated by $\mathcal{A}^N_0$.

Let $P_N: C^0(\mathcal{M}) \to L^2(\mathcal{M})$ be defined by $(P_N \phi)(\alpha) \triangleq \sum_{k=1}^N \mathbf{1}_{A_k}(\alpha) \phi(\alpha_k^N)$. 

\begin{lemma} \label{lem:nodal_interpolation_error}
For any $\phi \in H^2(\mathcal{M})$, the interpolation error satisfies
\begin{equation}
\|(I - P_N) \phi\|_{L^2} \le C_{cell} \delta_N^{1/2} \|\phi\|_{H^2},
\end{equation}
where $C_{cell} > 0$ is a constant independent of $N$.
\end{lemma}

\begin{proof}
See Appendix \ref{pf:nodal_interpolation_error}.
\end{proof}

\begin{lemma} \label{lem:semigroup_difference_spectral}
Let Assumption \ref{as:uniform_isotropic_graph} and Assumption \ref{asm:manifold_dim} hold. For a fixed finite terminal time $T>0$ and $f \in L^\infty(0,T; H^2(\mathcal{M}))$, it holds that
\begin{equation}
\sup_{t \in [0,T]} \sup_{\tau \in [0,t]} \|S_N(t-\tau) P_N f(\tau) - P_N S_0(t-\tau) f(\tau)\|_{L^2} \le \mathcal{O}\big(r_N^{2/5} + \delta_N^{1/2}\big).
\end{equation}
\end{lemma}

\begin{proof}
See Appendix \ref{pf:semigroup_difference_spectral}.
\end{proof}

\begin{theorem} \label{thm:error_bound_global_l2}
    Let the tracking error be defined as $\mathbf{e}_t^N = \bar{\mathbf{x}}_t^N - \mathbf{m}_t$. Let Assumptions \ref{as:uniform_isotropic_graph}, \ref{asm:kernel_lipschitz}, \ref{asm:general_degree}, and \ref{asm:manifold_dim} hold.  Then there exists a uniform time horizon $T^{**} > 0$, independent of $N$, such that for any $T \in (0, T^{**}]$ and any confidence level $\delta \in (0,1)$, the total tracking error satisfies the following bounds with probability at least $1-\delta$:
    \begin{enumerate}
        \item \textbf{Logarithmic Sparse Regime} If $d_{\min} = \Theta(\log N)$, the tracking error satisfies
        \begin{equation}
\sup_{t \in [0,T]} \|\mathbf{e}_t^N\|_{L^2} \le \mathcal{O}\left( \sqrt{\frac{\ln(2/\delta)}{\log N}} \right).
\end{equation}
        \item \textbf{Dense Regime} If $d_{\min} = \Theta(N)$, the tracking error satisfies
        \begin{equation}
\sup_{t \in [0,T]} \|\mathbf{e}_t^N\|_{L^2} \le \mathcal{O}\left( N^{-\gamma(p,s)} \right).
\end{equation}
        where $\gamma(p,s) \triangleq \min\left\{ \frac{2}{5p}, \frac{s}{p} - \frac{1}{2} \right\}$.
    \end{enumerate}
\end{theorem}
\begin{proof}
    Consider the high-probability event $\Omega_0 \subseteq \Omega_z$ defined by 
    \begin{equation}
\Omega_0 = \left\{ \omega \in \Omega_z \;\Big|\; \|\Phi_N^\omega\|_{\mathcal{L}(L^2)} \le C_{\max}(\delta) \right\} \cap \left\{ \omega \in \Omega_z \;\Big|\; \|\mathcal{E}_N\|_{H^s \to L^2} \le \epsilon_{op}(N, \delta/2) \right\}.
\end{equation}
    By Lemma \ref{lem:L2_norm_bound_operator}, $\mathbb{P}_z\left(\|\Phi_N^\omega\|_{\mathcal{L}(L^2)} > C_{\max}(\delta)\right) \le \delta/2$. Concurrently, Proposition \ref{prop:total_operator_error} ensures that $\mathbb{P}_z\left(\|\mathcal{E}_N\|_{H^s \to L^2} > \epsilon_{op}(N, \delta/2)\right) \le \delta/2$. Hence, $\mathbb{P}_z(\Omega_0) \ge 1 - \delta$. We proceed with the analysis conditional on a fixed $\omega \in \Omega_0$.

    Define the error $\boldsymbol{\epsilon}_t^\theta = \bar{\boldsymbol{\theta}}_t^N - \Theta_t$. The evolution of $\mathbf{e}_t^N$ involves $\boldsymbol{\epsilon}_t^\theta$ and is given by 
    \begin{equation}\label{eq:forward_error_global_l2}
    \begin{aligned}
        \mathbf{e}_t^N &= \left( S_N(t)\bar{\mathbf{x}}_0^N - S_0(t)\mathbf{m}_0 \right) + \int_0^t \left[ S_N(t-\tau) \left( c \Phi_N^\omega \bar{\mathbf{x}}_\tau^N - \frac{b^2}{r} \bar{\boldsymbol{\theta}}_\tau^N \right) - S_0(t-\tau) F(\tau) \right] \diff \tau\\
        &=S_N(t)(\bar{\mathbf{x}}_0^N - P_N \mathbf{m}_0) + (S_N(t) P_N - P_N S_0(t))\mathbf{m}_0 + (P_N - I) S_0(t)\mathbf{m}_0 \\
        &\quad + \int_0^t S_N(t-\tau) \left( c \Phi_N^\omega \mathbf{e}_\tau^N - \frac{b^2}{r} \boldsymbol{\epsilon}_\tau^\theta \right) \diff \tau + \int_0^t S_N(t-\tau) c \mathcal{E}_N \mathbf{m}_\tau \diff \tau \\
        &\quad + \int_0^t \big( (S_N(t-\tau) P_N - P_N S_0(t-\tau)) F(\tau) + (P_N - I) S_0(t-\tau) F(\tau) \\
        &\qquad+ S_N(t-\tau)(I - P_N)F(\tau) \big) \diff \tau,
    \end{aligned}
    \end{equation}
    where $F(\tau) = c \mathbf{G} \mathbf{m}_\tau - \frac{b^2}{r} \Theta_\tau$. 

    We first analyze the bound for $\boldsymbol{\epsilon}_\tau^\theta$ whose dynamics are given by
    \begin{align*}
        -\partial_t \boldsymbol{\epsilon}_t^\theta &= a \boldsymbol{\epsilon}_t^\theta + q_0(\mathbb{I} - H \Phi_N^\omega) \mathbf{e}_t^N - q_0 H \mathcal{E}_N \mathbf{m}_t, \\
        \boldsymbol{\epsilon}_T^\theta &= q_T(\mathbb{I} - H \Phi_N^\omega) \mathbf{e}_T^N - q_T H \mathcal{E}_N \mathbf{m}_T.
    \end{align*}
    The solution can be written as
    \begin{equation}\label{eq:epsilon_t}
         \boldsymbol{\epsilon}_t^\theta = e^{a(T-t)} \boldsymbol{\epsilon}_T^\theta + \int_t^T e^{a(s-t)} \left( q_0(\mathbb{I} - H \Phi_N^\omega) \mathbf{e}_s^N - q_0 H \mathcal{E}_N \mathbf{m}_s \right) \diff s.
    \end{equation}
    Taking the $L^2$-norm gives
    \begin{equation}
\begin{aligned}
            \|\boldsymbol{\epsilon}_t^\theta\|_{L^2} &\le e^{|a|(T-t)} \|\boldsymbol{\epsilon}_T^\theta\|_{L^2}  \\
            &\quad +\int_t^T e^{|a|(s-t)} \left( q_0 \|\mathbb{I} - H \Phi_N^\omega\|_{\mathcal{L}(L^2)} \|\mathbf{e}_s^N\|_{L^2} + q_0 |H| \|\mathcal{E}_N \mathbf{m}_s\|_{L^2} \right) \diff s.
       \end{aligned}
\end{equation}
    Since $H^2(\mathcal{M}) \hookrightarrow H^s(\mathcal{M})$ for $s < 2$, we have $\sup_{\tau \in [0,T]} \|\mathbf{m}_\tau\|_{H^s} \le M_{\mathbf{m}} < \infty$. Then $\|\mathcal{E}_N \mathbf{m}_s\|_{L^2} \le \epsilon_{op} \|\mathbf{m}_s\|_{H^s} \le \epsilon_{op} M_{\mathbf{m}}$. 
    
    Define $u(t) = \sup_{\tau \in [0,t]} \|\mathbf{e}_\tau^N\|_{L^2}$ for any $t \in [0,T]$ and let $C_{\Phi, I} = 1 + |H| C_{\max}(\delta) \ge \|\mathbb{I} - H \Phi_N^\omega\|_{\mathcal{L}(L^2)}$, the terminal condition is bounded by
    \begin{equation}
\|\boldsymbol{\epsilon}_T^\theta\|_{L^2} \le q_T C_{\Phi, I} u(T) + q_T |H| M_{\mathbf{m}} \epsilon_{op}.
\end{equation}
    Substituting this terminal bound into \eqref{eq:epsilon_t}, we obtain
    \begin{equation}
\|\boldsymbol{\epsilon}_t^\theta\|_{L^2} \le e^{|a|T} (q_T + q_0 T) C_{\Phi, I} u(T) + e^{|a|T} (q_T + q_0 T) |H| M_{\mathbf{m}} \epsilon_{op} \triangleq C_B u(T) + C_E \epsilon_{op}.
\end{equation}
    
     Next, we bound the last integral in \eqref{eq:forward_error_global_l2}. Assumption \ref{asm:kernel_lipschitz} ensures $\mathbf{m}_\tau, F(\tau) \in H^2(\mathcal{M})$. Lemma \ref{lem:semigroup_difference_spectral} provides 
    \begin{equation}
\sup_{\tau \in [0,T]} \|\big(S_N(t-\tau) P_N - P_N S_0(t-\tau)\big) F(\tau)\|_{L^2} \le \mathcal{O}(r_N^{2/5} + \delta_N^{1/2}).
\end{equation}
    Similarly, $\|(S_N(t) P_N - P_N S_0(t))\mathbf{m}_0\|_{L^2} \le \mathcal{O}(r_N^{2/5} + \delta_N^{1/2})$.
    For the remaining terms involving $(P_N - I)$, Lemma \ref{lem:nodal_interpolation_error} establishes $\|(I - P_N) \phi\|_{L^2} \le C_{cell} \delta_N^{1/2} \|\phi\|_{H^2}$. Therefore, these components are uniformly bounded by $\Delta_{geo} = \mathcal{O}(r_N^{2/5} + \delta_N^{1/2})$. Furthermore, by the assumption in Section 3.2, $\bar{\mathbf{x}}_0^N = P_N \mathbf{m}_0$, hence $S_N(t)(\bar{\mathbf{x}}_0^N - P_N \mathbf{m}_0) = 0$.

     We now substitute $\|\boldsymbol{\epsilon}_t^\theta\|_{L^2}$ and $\Delta_{geo}$ back into \eqref{eq:forward_error_global_l2}. Taking the $L^2$-norm of \eqref{eq:forward_error_global_l2}, applying $\|S_N(t)\|_{\mathcal{L}(L^2)} \le M e^{\omega_0 t}$, and defining $\omega_0^+ = \max\{\omega_0, 0\}$ establishes
    \begin{equation}
\|\mathbf{e}_t^N\|_{L^2} \le \Delta_{geo} + \int_0^t \left[ \bar{K}_1(\delta) \|\mathbf{e}_\tau^N\|_{L^2} + K_2 u(T) + K_3 \epsilon_{op} \right] \diff \tau,
\end{equation}
    where $\bar{K}_1(\delta) = M e^{\omega_0^+ T} |c| C_{\max}(\delta)$, $K_2 = M e^{\omega_0^+ T} \frac{b^2}{r} C_B$, and $K_3 = M e^{\omega_0^+ T} \left( \frac{b^2}{r} C_E + |c| M_{\mathbf{m}} \right)$ independent of $N$. Recall that $r_N = \mathcal{O}(N^{-1/p})$ and $\delta_N = \mathcal{O}(N^{-1/p})$, the combined deterministic error satisfies $\Delta_{geo} \le \mathcal{O}(N^{-\frac{2}{5p}} + N^{-\frac{1}{2p}}) = \mathcal{O}(N^{-\frac{2}{5p}})$. Thus,
    \begin{equation}
\|\mathbf{e}_t^N\|_{L^2} \le \mathcal{O}(N^{-\frac{2}{5p}}) + \int_0^t \left[ \bar{K}_1(\delta) \|\mathbf{e}_\tau^N\|_{L^2} + K_2 u(T) + K_3 \epsilon_{op} \right] \diff \tau.
\end{equation}

    Applying Grönwall's inequality and taking the supremum over $t \in [0,T]$ yields
    \begin{equation}
u(T) \left( 1 - K_2 \frac{e^{\bar{K}_1(\delta) T} - 1}{\bar{K}_1(\delta)} \right) \le \mathcal{O}(N^{-\frac{2}{5p}}) e^{\bar{K}_1(\delta) T} + K_3 \epsilon_{op} \frac{e^{\bar{K}_1(\delta) T} - 1}{\bar{K}_1(\delta)}.
\end{equation}
    Since $\bar{K}_1(\delta)$ is uniform with respect to $N$, the limit $\lim_{T \to 0^+} \frac{e^{\bar{K}_1(\delta) T} - 1}{\bar{K}_1(\delta)} = 0$ guarantees the existence of a uniform horizon $T^{**} > 0$ such that $1 - K_2 \frac{e^{\bar{K}_1(\delta) T} - 1}{\bar{K}_1(\delta)} \ge 1/2$. For any $T \in (0, T^{**}]$,
    \begin{equation}
u(T) \le \mathcal{O}(N^{-\frac{2}{5p}}) + \tilde{C}_{sys}(\delta) \epsilon_{op},
\end{equation}
    where $\tilde{C}_{sys}(\delta) \triangleq 2 K_3 \frac{e^{\bar{K}_1(\delta) T} - 1}{\bar{K}_1(\delta)}$.

    By Proposition \ref{prop:total_operator_error}, we have $\epsilon_{op} \le \mathcal{O}\left( \sqrt{\ln(2/\delta)/d_{\min}} + N^{-\kappa(p,s)} \right)$. 
    
    If $d_{\min} = \Theta(\log N)$, the stochastic sampling term $\mathcal{O}(\sqrt{\ln(2/\delta)}(\log N)^{-1/2})$ is asymptotically dominant over all polynomial decay rates, yielding $\sup_{t \in [0,T]} \|\mathbf{e}_t^N\|_{L^2} \le \mathcal{O}(\sqrt{\ln(2/\delta)}(\log N)^{-1/2})$.
    
    If $d_{\min} = \Theta(N)$, the stochastic sampling term converges at $\mathcal{O}(\sqrt{\ln(2/\delta)}N^{-1/2})$. Hence, the combined tracking error is governed by $\sup_{t \in [0,T]} \|\mathbf{e}_t^N\|_{L^2} \le \mathcal{O}(N^{-\gamma(p,s)})$.
\end{proof}

\section{Numerical Experiments}\label{sec:simulation}

This section {provides Monte Carlo simulations to validate} the finite-population approximations. The state space is a one-dimensional manifold $\mathcal{M} = [0, 1]$ equipped with periodic boundary conditions. The underlying interaction network is generated by a deterministic core graph superimposed with a stochastic global graph.

The population size is $N=400$ and the time horizon is $T=3.0$, discretized with a time step $\Delta t = 10^{-3}$. The dynamics coefficients are $a=-1.2$, $b=1.0$, {$c=1.0$, $\sigma=0.5$}, and $\gamma=0.001$. The cost function parameters are $q_0=0.5$, $r=1.0$, $q_T=0.5$, $H=1.0$, and $\eta=0.0$. The initial condition is $x_0(\alpha) = \sin(4\pi\alpha) + 1.0$.

The non-local interactions are drawn from the probability measure associated with the continuous kernel {$K(\alpha,\beta) = 1.0 + 0.3\cos(2\pi(\alpha-\beta))$}.

Two asymptotic network regimes are compared. In the sparse regime, the nodal degrees are sampled from a {truncated} Poisson distribution with expectation {$\lceil 2\ln N \rceil = 12$}. In the dense regime, the degrees are sampled uniformly from the interval $[280, 380]$.

The continuum limits are calculated by solving the corresponding FBPDE system. The tracking performance is measured by the $L^2$ error {$\|\mathbf{e}_t^N\|_{L^2} = \|\bar{\mathbf{x}}_t^N - \mathbf{m}_t\|_{L^2}$}.

{Figure \ref{fig:error_convergence} plots} the evolution of $\|\mathbf{e}_t^N\|_{L^2}$ for both regimes over 50 simulation trials. {Both networks exhibit similar behaviors with the tracking error decreasing over time. However, the sparse network maintains a higher error magnitude throughout the time horizon.}

To validate the theoretical bounds established in Theorem \ref{thm:error_bound_global_l2}, the stochastic topological error is evaluated by isolating it from the deterministic spatial discretization error $\mathcal{O}(N^{-\gamma(p,s)})$. By fixing the network size $N=400$, the $\mathcal{O}(N^{-\gamma(p,s)})$ term is constrained to a constant baseline. The expected graph degree $d_{avg}$ is then varied across multiple trials to observe the transition from the sparse to the dense regime. As illustrated in Fig. \ref{fig:scaling_law}, the tracking error $\|\mathbf{e}_t^N\|_{L^2}$ is plotted. As $d_{avg}$ increases, the actual error curve asymptotically aligns with the theoretical reference line of slope $-0.5$. This confirmation matches with the operator concentration bound $\mathcal{O}(d_{\min}^{-1/2})$ derived in Proposition \ref{prop:total_operator_error}.

\begin{figure}[htbp]
    \centering
    \begin{subfigure}[b]{0.48\textwidth}
        \centering
        \includegraphics[width=\textwidth]{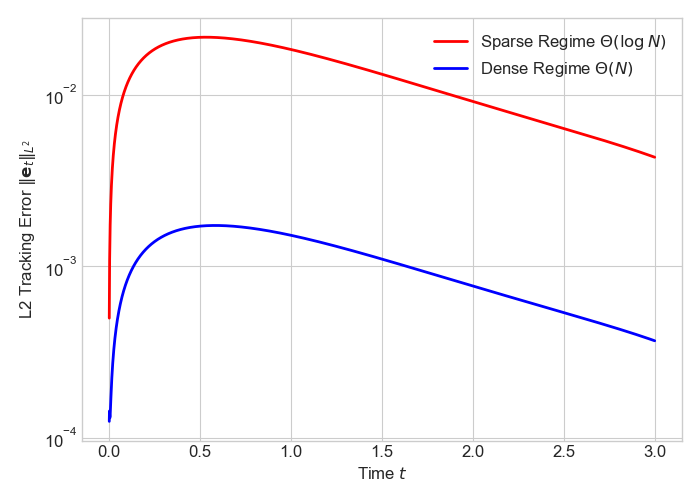}
        \caption{Evolution of $\|\mathbf{e}_t^N\|_{L^2}$.}
        \label{fig:error_convergence}
    \end{subfigure}
    \hfill 
    \begin{subfigure}[b]{0.48\textwidth}
        \centering
        \includegraphics[width=\textwidth]{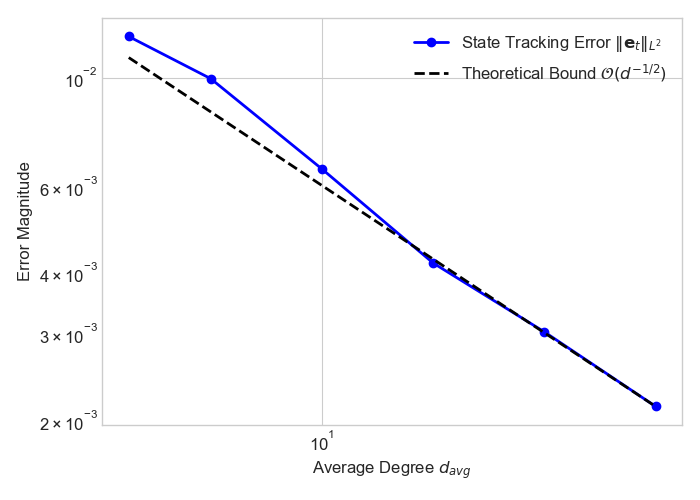}
        \caption{Error versus average degree.}
        \label{fig:scaling_law}
    \end{subfigure}
    
    \caption{(a) Time evolution of the tracking error for both sparse and dense regimes. (b) Scaling law of the time-averaged error with respect to the network degree.}
    \label{fig:numerical_results}
\end{figure}

Fig. \ref{fig:trajectory_shading} records the local mean states at $\alpha \in \{0.125, 0.5, 0.875\}$. The shaded regions indicate the {$\pm 2$ standard deviation} intervals. 

\begin{figure}[htbp]
    \centering
    \includegraphics[width=0.96\textwidth]{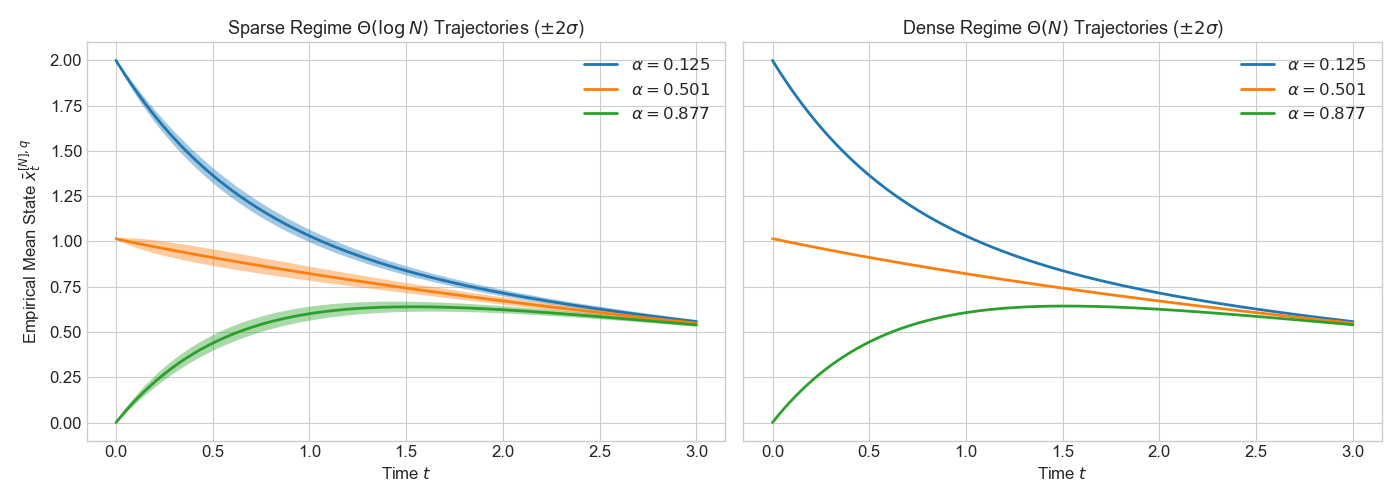}
    \caption{Time evolution of the local mean state for selected coordinates. Shaded regions represent {$\pm 2$ standard deviation} intervals over 50 trials. Left: Sparse network. Right: Dense network.}
    \label{fig:trajectory_shading}
\end{figure}
\section{Conclusion}\label{sec:conclusion}

This paper {investigated} Linear-Quadratic-Gaussian games formulated on a hybrid graph structure. The network topology combines a deterministic local geometric core and a stochastic global periphery. The continuum limit of the game is characterized by a system of Forward-Backward PDEs governing $\mathbf{m}_t$ and $\Theta_t$.

{Non-asymptotic high-probability upper bounds are established for the approximation error $\mathbf{e}_t^N$} between the finite-population {empirical state} $\bar{\mathbf{x}}_t^N$ and {the continuum limit state} $\mathbf{m}_t$. For the sparse regime with out-degrees scaling as $\Theta(\log N)$, the error is bounded by $\mathcal{O}((\log N)^{-1/2})$. For the dense regime scaling as $\Theta(N)$, the error is bounded by $\mathcal{O}(N^{-\gamma(p,s)})$. 

Future research could extend the stochastic sampling mechanism of the global graph to expander graph structures to model interactions in {other complex network topologies}.

\section*{Acknowledgments}
The author would like to acknowledge the assistance of Prof. Peter E. Caines in providing valuable suggestions and insightful comments.

\section*{Declarations}

During the preparation of this manuscript, the authors utilized ChatGPT to assist with language polishing and textual error correction. All mathematical results, computational experiments, and conclusions were verified by the authors. The authors assume responsibility for all content.

\appendix
\section{Operator Error Decomposition}
Define the \textit{piece-wise constant function} $\bar{\mathbf{x}}_t^N$ on $\M$ as $\bar{\mathbf{x}}_t^N(\alpha)=\bar{x}^{[N],q}_t$ for any $\alpha\in A_q$.

Define the discretized integral operator $\mathbf{G}_N$ by
\begin{equation} \label{eq:def_G_N}
 [\mathbf{G}_N f](\alpha) \triangleq \sum_{k=1}^N \mathbf{1}_{A_k}(\alpha) \int_{\mathcal{M}} g(\alpha_k, \beta) f(\beta) \diff V_g(\beta).
\end{equation}

Define the intermediate expectation operator $\tilde{\mathbf{G}}_N: L^2(\mathcal{M}) \to L^2(\mathcal{M})$ by
\begin{equation}
    [\tilde{\mathbf{G}}_N f](\alpha) \triangleq \sum_{q=1}^N \mathbf{1}_{A_q}(\alpha) \sum_{k=1}^N G^{[N]}_{qk} f_{A_k},
\end{equation}
where $f_{A_k} \triangleq \frac{1}{\mathrm{Vol}(A_k)} \int_{A_k} f(\beta) \diff V_g(\beta)$. By definition, $\mathbb{E}_z[\Phi_N^\omega f] = \tilde{\mathbf{G}}_N f$.

\begin{lemma}\label{lem:W_N_weak_convergence}
    For any test function $f \in H^s(\mathcal{M})$ with $s > p/2$, the weak convergence error operator $\mathcal{W}_N \triangleq \tilde{\mathbf{G}}_N - \mathbf{G}_N$ satisfies 
    \begin{equation}
\|\mathcal{W}_N\|_{H^s \to L^2} \le \mathcal{O}(N^{\frac{1}{2} - \frac{s}{p}}).
\end{equation}
\end{lemma}

\begin{proof}
    Define the cell-averaging projection operator $\Pi_N: L^2(\mathcal{M}) \to L^2(\mathcal{M})$ acting on any function $f$ by
    \begin{equation}
(\Pi_N f)(\beta) \triangleq \sum_{k=1}^N \mathbf{1}_{A_k}(\beta) f_{A_k} = \sum_{k=1}^N \mathbf{1}_{A_k}(\beta) \frac{1}{\mathrm{Vol}(A_k)} \int_{A_k} f(\zeta) \diff V_g(\zeta).
\end{equation}
    Substituting this into the definition of $\mathbf{G}_N$, we observe that 
    \begin{equation}
[\mathbf{G}_N (\Pi_N f)](\alpha_q) = \sum_{k=1}^N \int_{A_k} g(\alpha_q, \beta) f_{A_k} \diff V_g(\beta) = \sum_{k=1}^N G^{[N]}_{qk} f_{A_k} = [\tilde{\mathbf{G}}_N f](\alpha_q).
\end{equation}
    Consequently, the weak convergence error operator can be factorized as $ \mathcal{W}_N f = \mathbf{G}_N (\Pi_N f - f)$.
    Taking the $L^2$-norm yields
    \begin{equation}
\|\mathcal{W}_N f\|_{L^2} \le C_{G} \|\Pi_N f - f\|_{L^2}.
\end{equation}
for sufficiently large $N$ with some constant $C_G$. For $f \in H^s(\mathcal{M})$ with $s > p/2$, the fractional Sobolev embedding theorem guarantees $H^s(\mathcal{M}) \hookrightarrow C^{s - p/2}(\mathcal{M})$. Then there exists a constant $C_{emb} > 0$ such that
    \begin{equation}
\sup_{\beta \in A_k} |f(\beta) - f_{A_k}| \le \sup_{\beta, \zeta \in A_k} |f(\beta) - f(\zeta)| \le C_{emb} \delta_N^{s - p/2} \|f\|_{H^s}.
\end{equation}
    Hence, $\|\Pi_N f - f\|_{L^2}^2$ is bounded by
    \begin{equation}
\|\Pi_N f - f\|_{L^2}^2 = \sum_{k=1}^N \int_{A_k} |f(\beta) - f_{A_k}|^2 \diff V_g(\beta) \le  C_{emb}^2 \delta_N^{2s - p} \|f\|_{H^s}^2.
\end{equation}
    Substituting $\delta_N = \mathcal{O}(N^{-1/p})$ yields
    \begin{equation}
\|\Pi_N f - f\|_{L^2} \le C_{emb}  N^{\frac{1}{2} - \frac{s}{p}} \|f\|_{H^s}.
\end{equation}
    Therefore, the operator norm is bounded by $\|\mathcal{W}_N\|_{H^s \to L^2} \le \mathcal{O}(N^{\frac{1}{2} - \frac{s}{p}})$. 
\end{proof}

\begin{lemma}\label{lem:G_N_uniform_convergence}
    Under Assumptions \ref{asm:kernel_lipschitz} and \ref{asm:manifold_dim}, the operator difference between $\mathbf{G}_N$ and $\mathbf{G}$ is bounded by
    \begin{equation}
\|\mathbf{G}_N - \mathbf{G}\|_{\mathcal{L}(L^2)} \le \mathcal{O}(\delta_N^{1/2}) = \mathcal{O}(N^{-\frac{1}{2p}}).
\end{equation}
\end{lemma}

\begin{proof}
    For any test function $f \in L^2(\mathcal{M})$ and any point $\alpha \in A_k$,
    \begin{equation}
|(\mathbf{G}_N f)(\alpha) - (\mathbf{G} f)(\alpha)| \le \int_{\mathcal{M}} |g(\alpha_k, \beta) - g(\alpha, \beta)| |f(\beta)| \diff V_g(\beta).
\end{equation}
    By Assumption \ref{asm:kernel_lipschitz}, $\sup_\beta \|g(\cdot, \beta)\|_{H^2(\mathcal{M})} \le C_g$. Since the manifold dimension is $p \le 3$, the Sobolev embedding theorem gives $H^2(\mathcal{M}) \hookrightarrow C^{0,1/2}(\mathcal{M})$. Then there exists a constant $C_{emb} > 0$ independent of $N$ such that
    \begin{equation}
|g(\alpha_k, \beta) - g(\alpha, \beta)| \le C_{emb} \|g(\cdot, \beta)\|_{H^2(\mathcal{M})} d_g(\alpha_k, \alpha)^{1/2} \le C_{emb} C_g \delta_N^{1/2}.
\end{equation}
    Substituting this bound into the integral and applying the Cauchy-Schwarz inequality over the manifold yields
    \begin{equation}
|(\mathbf{G}_N f)(\alpha) - (\mathbf{G} f)(\alpha)| \le C_{emb} C_g \delta_N^{1/2} \int_{\mathcal{M}} |f(\beta)| \diff V_g(\beta) \le C_{emb} C_g \delta_N^{1/2}  \|f\|_{L^2}.
\end{equation}
    Consequently, the $L^2$-norm is bounded by
    \begin{equation}
\|\mathbf{G}_N f - \mathbf{G} f\|_{L^2} \le C_{emb} C_g \delta_N^{1/2} \|f\|_{L^2}.
\end{equation}
     This completes the proof.
\end{proof}

Recall the operator difference $\mathcal{E}_N \triangleq \Phi_N^\omega - \mathbf{G}$. It can be decomposed into the stochastic sampling error $\mathcal{D}_N^\omega$, the weak convergence error $\mathcal{W}_N$, and the discretization error $\mathcal{R}_N$:
\begin{equation}
    \mathcal{E}_N = \mathcal{D}_N^\omega + \mathcal{W}_N + \mathcal{R}_N \triangleq (\Phi_N^\omega - \tilde{\mathbf{G}}_N) + (\tilde{\mathbf{G}}_N - \mathbf{G}_N) + (\mathbf{G}_N - \mathbf{G}).
\end{equation}

\begin{proposition}\label{prop:total_operator_error}
    Let Assumptions \ref{asm:kernel_lipschitz}, \ref{asm:general_degree}, and \ref{asm:manifold_dim} hold. For any confidence level $\delta \in (0,1)$, the operator norm bound $\epsilon_{op} \triangleq \|\mathcal{E}_N\|_{H^s \to L^2}$ satisfies
    \begin{equation}
\epsilon_{op} \le \mathcal{O}\left( \sqrt{\frac{\ln(2/\delta)}{d_{\min}}} + N^{-\kappa(p,s)} \right)
\end{equation}
    with probability at least $1-\delta$, where $\kappa(p,s) \triangleq \min\left\{\frac{1}{2p}, \frac{s}{p} - \frac{1}{2}\right\}$.
\end{proposition}

\begin{proof}
    The triangle inequality gives
    \begin{equation}
\|\mathcal{E}_N\|_{H^s \to L^2} \le \|\mathcal{D}_N^\omega\|_{H^s \to L^2} + \|\mathcal{W}_N\|_{H^s \to L^2} + \|\mathcal{R}_N\|_{H^s \to L^2}.
\end{equation}
    For $\mathcal{R}_N = \mathbf{G}_N - \mathbf{G}$, Lemma \ref{lem:G_N_uniform_convergence} establishes
    \begin{equation}
\|\mathcal{R}_N\|_{H^s \to L^2} \le \|\mathcal{R}_N\|_{\mathcal{L}(L^2)} = \mathcal{O}(N^{-\frac{1}{2p}}).
\end{equation}
    For $\mathcal{W}_N=\tilde{\mathbf{G}}_N - \mathbf{G}_N$, Lemma \ref{lem:W_N_weak_convergence} yields
    \begin{equation}
\|\mathcal{W}_N\|_{H^s \to L^2} \le \mathcal{O}(N^{\frac{1}{2} - \frac{s}{p}}).
\end{equation}
    For $\mathcal{D}_N^\omega=\Phi_N^\omega - \tilde{\mathbf{G}}_N$, Theorem \ref{thm:op-norm-main-tight} provides the Bernstein tail bound
    \begin{equation}
\|\mathcal{D}_N^\omega\|_{H^s \to L^2} \le \mathcal{O}\left( \sqrt{\frac{\ln(2/\delta)}{d_{\min}}} + \frac{\ln(2/\delta)}{\sqrt{N}} \right).
\end{equation}
    Under Assumption \ref{asm:manifold_dim}, $p \le 3$, which ensures $1/(2p) \ge 1/6 > 0$. The condition $s > p/2$ ensures $s/p - 1/2 > 0$. Therefore, $\kappa(p,s)$ is positive. Since $\delta$ is a fixed confidence level, the term $\ln(2/\delta) N^{-1/2}$ is bounded by $\mathcal{O}(N^{-1/2})$. Because $p \ge 1$, we have $1/(2p) \le 1/2$, which implies $\kappa(p,s) \le 1/2$. Thus, the algebraic tail $\mathcal{O}(N^{-1/2})$ is asymptotically absorbed by or equivalent to $\mathcal{O}(N^{-\kappa(p,s)})$. Combining all error terms establishes the final bound
    \begin{equation}
\epsilon_{op} \le \mathcal{O}\left( \sqrt{\frac{\ln(2/\delta)}{d_{\min}}} + N^{-\kappa(p,s)} \right).
\end{equation}
\end{proof}

\section{Proof of Lemma \ref{lem:nodal_interpolation_error}}\label{pf:nodal_interpolation_error}
\begin{proof}
For a compact Riemannian manifold $\mathcal{M}$ of dimension $p \le 3$, the Sobolev embedding theorem provides the continuous embedding $H^2(\mathcal{M}) \hookrightarrow C^{0, 1/2}(\mathcal{M})$. There exists a constant $C_{cell}  > 0$ such that for any $\phi \in H^2(\mathcal{M})$ and any $x, y \in \mathcal{M}$,
\begin{equation}
|\phi(x) - \phi(y)| \le C_{cell}  d_g(x, y)^{1/2} \|\phi\|_{H^2}.
\end{equation}
For any point $\alpha$ within a specific cell $A_k$, we have $d_g(\alpha, \alpha_k^N) \le \delta_N$. Therefore,
\begin{equation}
|\phi(\alpha) - (P_N \phi)(\alpha)| = |\phi(\alpha) - \phi(\alpha_k^N)| \le C_{cell}  \delta_N^{1/2} \|\phi\|_{H^2}.
\end{equation}
Integrating over $A_k$ gives
\begin{equation}
\int_{A_k} |\phi(\alpha) - (P_N \phi)(\alpha)|^2 dV_g(\alpha) \le C_{cell} ^2 \delta_N \|\phi\|_{H^2}^2 \mathrm{Vol}(A_k).
\end{equation}
Hence,
\begin{equation}
\|(I - P_N) \phi\|_{L^2}^2 = \sum_{k=1}^N \int_{A_k} |\phi(\alpha) - (P_N \phi)(\alpha)|^2 dV_g(\alpha) \le C_{cell} ^2 \delta_N \|\phi\|_{H^2}^2 .
\end{equation}
\end{proof}

\section{Proof of Lemma \ref{lem:semigroup_difference_spectral}}\label{pf:semigroup_difference_spectral}
\begin{proof}
Fix $\tau \in [0, t]$ and denote $\psi = f(\tau) \in H^2(\mathcal{M})$. Let $\{\lambda_m, \phi_m\}_{m=0}^\infty$ be the eigenvalues and corresponding orthonormal eigenfunctions of the negative Laplace-Beltrami operator $-\Delta_g$ on $\mathcal{M}$, arranged in ascending order. For any $\lambda > 0$, the spectral truncation operator $\mathcal{T}_\lambda \colon L^2(\mathcal{M}) \to L^2(\mathcal{M})$ is defined for any $f \in L^2(\mathcal{M})$ by
\begin{equation}
\mathcal{T}_\lambda f \triangleq \sum_{\lambda_m \le \lambda} \langle f, \phi_m \rangle_{L^2} \phi_m.
\end{equation}

We decompose the test function $\psi$ into two components
\begin{equation}
\psi = \psi_\lambda + \psi^\perp \triangleq \mathcal{T}_\lambda \psi + (I - \mathcal{T}_\lambda) \psi.
\end{equation}
Based on the spectral expansion $\psi = \sum_{m=0}^\infty c_m \phi_m$ with $c_m = \langle \psi, \phi_m \rangle_{L^2}$, the Sobolev norm is $\|\psi\|_{H^s}^2 = \sum_{m=0}^\infty (1+\lambda_m)^s c_m^2$. Then, for $k \ge 2$,
\begin{equation}
\|\psi_\lambda\|_{H^k}^2 = \sum_{\lambda_m \le \lambda} (1+\lambda_m)^{k-2} (1+\lambda_m)^2 c_m^2 \le (1+\lambda)^{k-2} \|\psi\|_{H^2}^2.
\end{equation}
For $\psi_\perp$, its $L^2$-norm is bounded by
\begin{equation}
\begin{aligned}
\|\psi_\perp\|_{L^2}^2 &= \sum_{\lambda_m > \lambda} (1+\lambda_m)^{-2} (1+\lambda_m)^2 c_m^2 \\
&\le (1+\lambda)^{-2} \sum_{\lambda_m > \lambda} (1+\lambda_m)^2 c_m^2 \le (1+\lambda)^{-2} \|\psi\|_{H^2}^2.
\end{aligned}
\end{equation}
Therefore, for $k \ge 2$,
\begin{equation}
\|\psi_\lambda\|_{H^k} \le (1+\lambda)^{(k-2)/2} \|\psi\|_{H^2}, \quad \|\psi_\perp\|_{L^2} \le (1+\lambda)^{-1} \|\psi\|_{H^2}.
\end{equation}

Define the total approximation error as $e \triangleq \|S_N(t-\tau) P_N \psi - P_N S_0(t-\tau) \psi\|_{L^2}$. By the triangle inequality,
\begin{equation}\label{eq:e_total}
      e \le \|S_N(t-\tau) P_N \psi_\lambda - P_N S_0(t-\tau) \psi_\lambda\|_{L^2} + \|S_N(t-\tau) P_N \psi^\perp\|_{L^2} + \|P_N S_0(t-\tau) \psi^\perp\|_{L^2}.
\end{equation}
\begin{itemize}
    \item For the last term, recall that $H^2(\mathcal{M}) \hookrightarrow C^{0, 1/2}(\mathcal{M})$ for $p \le 3$. By Lemma \ref{lem:nodal_interpolation_error},
\begin{equation}
\|(I - P_N) S_0(t-\tau) \psi^\perp\|_{L^2} \le C_{cell} \delta_N^{1/2} \|S_0(t-\tau) \psi^\perp\|_{H^2}\le C_{cell} \delta_N^{1/2} e^{|a|T}\| \psi^\perp\|_{H^2}.
\end{equation}
Hence,
\begin{align*}
       \|P_N S_0(t-\tau) \psi^\perp\|_{L^2} \le& \|S_0(t-\tau) \psi^\perp\|_{L^2} + \|(I - P_N) S_0(t-\tau) \psi^\perp\|_{L^2} \\
        \le& e^{|a|T}(1+\lambda)^{-1} \|\psi\|_{H^2} + C_{cell} \delta_N^{1/2} e^{|a|T}\|\psi\|_{H^2}.
\end{align*}

\item Similarly, for the second term, one has
\begin{align*}
     \|S_N(t-\tau) P_N \psi^\perp\|_{L^2} &\le M e^{\omega_0^+ T} \|P_N \psi^\perp\|_{L^2} \\
     &\le M e^{\omega_0^+ T} \big(e^{|a|T}(1+\lambda)^{-1} + C_{cell} \delta_N^{1/2}\big) \|\psi\|_{H^2}.
\end{align*}

\item For the first term in terms of $\psi_\lambda$, the spectral projection ensures $\psi_\lambda \in C^\infty(\mathcal{M})$. Consequently,
\begin{equation}\label{eq:lem5_3_eq}
      \hspace{-0.5cm}S_N(t-\tau) P_N \psi_\lambda - P_N S_0(t-\tau) \psi_\lambda = \int_0^{t-\tau} S_N(t-\tau-s) \Big(\mathcal{A}^N_0 P_N - P_N \mathcal{A}_0\Big) S_0(s) \psi_\lambda \diff s.
\end{equation}
Let $\phi = S_0\psi_\lambda$. Recall the definition of $\mathbf{L}_N$,
\begin{equation}
(\mathbf{L}_N P_N \phi)(\alpha_i^N) = \frac{1}{|\mathcal{N}_i|}\sum_{j\in\mathcal{N}_i} w_{ij}^{(N)} \big( (P_N \phi)_{A_j} - (P_N \phi)_{A_i} \big).
\end{equation}
Since $P_N \phi$ takes the constant value $\phi(\alpha_k^N)$ on each cell $A_k$, one has $(P_N \phi)_{A_k} = \phi(\alpha_k^N)$. Then,
\begin{equation}
[\mathcal{A}^N_0 P_N \phi - P_N \mathcal{A}_0 \phi](\alpha_i^N) = \gamma \left( (\mathbf{L}_N P_N \phi)(\alpha_i^N) - \frac{1}{2}\Delta_g \phi(\alpha_i^N) \right).
\end{equation}
In the Riemannian normal coordinates centered at $\alpha_i^N\in \M$, the geodesic displacement vector in the tangent space $T_{\alpha_i^{N}}\mathcal{M}$ is $v_{ij} = \exp_{\alpha_i^N}^{-1}(\alpha_j^N)$ where $\alpha_j^N\in \M$. The Taylor expansion is
\begin{equation}
\phi(\alpha_j^N) - \phi(\alpha_i^N) = \nabla \phi(\alpha_i^N) \cdot v_{ij} + \frac{1}{2} v_{ij}^\top \nabla^2 \phi(\alpha_i^N) v_{ij} + R(v_{ij}),
\end{equation}
where $|R(v_{ij})| \le C \|\phi\|_{C^3} \|v_{ij}\|^3$ (see \cite{zhang2026IFAC} for detailed proof). Then applying Assumption \ref{as:uniform_isotropic_graph} provides \cite{zhang2026IFAC}
\begin{equation}
\begin{aligned}
    (\mathbf{L}_N P_N \phi)(\alpha_i^N)  - \frac{1}{2}\Delta_g \phi(\alpha_i^N) &= \nabla \phi \cdot \mathcal{O}(r_N) + \mathcal{O}(r_N) \|\phi\|_{C^3}\\
    &= \mathcal{O}(r_N) \|\phi\|_{C^2} + \mathcal{O}(r_N) \|\phi\|_{C^3}\\
    &= \mathcal{O}(r_N)\|\phi\|_{C^3}
\end{aligned}
\end{equation}
due to $\|\phi\|_{C^2} \le \|\phi\|_{C^3}$. Since $H^5(\mathcal{M}) \hookrightarrow C^3(\mathcal{M})$ for $p \le 3$, there exists a constant $C^\prime > 0$ such that
\begin{equation}
\left\| \Big(\mathcal{A}^N_0 P_N - P_N \mathcal{A}_0\Big) \phi \right\|_{L^2} \le C^\prime r_N \|\phi\|_{H^5}.
\end{equation}
Applying $\|\phi\|_{H^5} = \|S_0 \psi_\lambda\|_{H^5} \le e^{|a|T}(1+\lambda)^{3/2} \|\psi\|_{H^2}$ to \eqref{eq:lem5_3_eq}:
\begin{equation}
\|S_N(t-\tau) P_N \psi_\lambda - P_N S_0(t-\tau) \psi_\lambda\|_{L^2} \le M e^{\omega_0^+ T} T C^\prime r_N e^{|a|T}(1+\lambda)^{3/2} \|\psi\|_{H^2}.
\end{equation}
\end{itemize}

Finally, combining the bounds of three terms in \eqref{eq:e_total} gives
\begin{equation}
e \le C \Big( (1+\lambda)^{-1} + \delta_N^{1/2} + r_N (1+\lambda)^{3/2} \Big) \|\psi\|_{H^2},
\end{equation}
where $C$ is a constant independent of $N$ and $\lambda$. To obtain the optimal convergence rate, we balance the two terms dependent on the threshold $\lambda$ by equating $(1+\lambda)^{-1} = r_N (1+\lambda)^{3/2}$. This yields the optimal cut-off threshold $1+\lambda = r_N^{-2/5}$.

Substituting $1+\lambda = r_N^{-2/5}$ back into the total error bound provides
\begin{equation}
e \le C \Big( 2 r_N^{2/5} + \delta_N^{1/2} \Big) \|\psi\|_{H^2}.
\end{equation}
Taking the supremum over $\tau \in [0,t]$ and $t \in [0,T]$ establishes the uniform bound $\mathcal{O}(r_N^{2/5} + \delta_N^{1/2})$. This completes the proof.
\end{proof}
\section{Concentration of the Sampling Operator on Manifolds}
\label{apdix:norm_bound}

This section establishes a tail bound for the operator norm of $\mathcal{D}_N^\omega: H^s(\mathcal{M}) \to L^2(\mathcal{M})$ for $s > p/2$. Let $C_S > 0$ be the Sobolev embedding constant satisfying $\|f\|_{L^\infty} \le C_S \|f\|_{H^s}$.

Let $\mathcal{H}_{HS}$ be the space of Hilbert-Schmidt operators from $H^s(\mathcal{M})$ to $L^2(\mathcal{M})$, equipped with the norm $\| \cdot \|_{HS}$. Note that the operator norm is bounded by $\| \cdot \|_{H^s \to L^2} \le \| \cdot \|_{HS}$.

\begin{lemma}[Hilbert Space Vector Bernstein Inequality \cite{pinelis1994optimum}]
\label{thm:hilbert-bernstein}
Let $\mathcal{H}$ be a separable Hilbert space. Let $\{X_k\}_{k=1}^N$ be a finite sequence of independent random vectors in $\mathcal{H}$. Assume that for each $k$, $\mathbb{E}[X_k] = 0$ and there exists a deterministic constant $M_k > 0$ such that $\|X_k\|_{\mathcal{H}} \le M_k$ almost surely. Define $M = \max_k M_k$ and let the variance parameter $\sigma^2$ satisfy
\begin{equation}
\sum_{k=1}^N \mathbb{E} \left[ \|X_k\|_{\mathcal{H}}^2 \right] \le \sigma^2.
\end{equation}
Then, for all $t > 0$, the following tail bound holds
\begin{equation}
\mathbb{P}_z \left( \left\| \sum_{k=1}^N X_k \right\|_{\mathcal{H}} \ge t \right) \le 2 \exp\left( -\frac{t^2/2}{\sigma^2 + M t/3} \right).
\end{equation}
\end{lemma}

For any $q \in V^N$, let $\{L_j^{[N],q}\}_{j=1}^{d_q}$ be drawn independently from $G^{[N]}_{q,\cdot}$ and recall the integral average $f_{A_k} = \frac{1}{\mathrm{Vol}(A_k)} \int_{A_k} f(\beta) \diff V_g(\beta)$. By the Riesz representation theorem, there exists a unique representative random vector $\tilde{\eta}_q \in H^s(\mathcal{M})$ satisfying
\begin{equation}
    \langle \tilde{\eta}_q, f \rangle_{H^s} = \frac{1}{d_q} \sum_{j=1}^{d_q} f_{A_{L_j^{[N],q}}} - \sum_{k=1}^N G^{[N]}_{qk} f_{A_k}
\end{equation}
for all $f \in H^s(\mathcal{M})$.

\begin{lemma}
\label{lem:eta_bounded_tight}
Assume $s > p/2$. For each $q \in V^N$, the representative random vector $\tilde{\eta}_q \in H^s(\mathcal{M})$ satisfies $\|\tilde{\eta}_q\|_{H^s} \le 2C_S$ almost surely. The expectation of its squared norm satisfies
\begin{equation}
    \mathbb{E}_z \left[ \|\tilde{\eta}_q\|_{H^s}^2 \right] \le \frac{C_S^2}{d_q}.
\end{equation}
\end{lemma}

\begin{proof}
For any cell $A_k$, let $\psi_k \in H^s(\mathcal{M})$ satisfy $\langle \psi_k, f \rangle_{H^s} = f_{A_k}$. By the Sobolev embedding theorem, $|f_{A_k}| \le C_S \|f\|_{H^s}$, which implies $\|\psi_k\|_{H^s} \le C_S$. The vector $\tilde{\eta}_q$ can be written as
\begin{equation}
\tilde{\eta}_q = \frac{1}{d_q} \sum_{j=1}^{d_q} Y_j^q,
\end{equation}
where $Y_j^q = \psi_{L_j^{[N],q}} - \mathbb{E}_z[\psi_{L_1^{[N],q}}]$. The sequence $\{Y_j^q\}_{j=1}^{d_q}$ consists of independent, identically distributed, and zero-mean random vectors in $H^s(\mathcal{M})$. 

Since $\|\psi_k\|_{H^s} \le C_S$, the triangle inequality yields $\|Y_j^q\|_{H^s} \le 2C_S$. Therefore, $\|\tilde{\eta}_q\|_{H^s} \le 2C_S$. In addition, $\mathbb{E}_z [ \|Y_j^q\|_{H^s}^2 ] \le \mathbb{E}_z [ \|\psi_{L_j^{[N],q}}\|_{H^s}^2 ] \le C_S^2$. Consequently,
\begin{equation}
\mathbb{E}_z [ \|\tilde{\eta}_q\|_{H^s}^2 ] = \frac{1}{d_q^2} \sum_{j=1}^{d_q} \mathbb{E}_z [ \|Y_j^q\|_{H^s}^2 ]\le C_S^2 / d_q.
\end{equation}
\end{proof}

The operator action on $f \in H^s(\mathcal{M})$ is
\begin{equation} \label{eq:def_E_N_derived}
    [\mathcal{D}_N^\omega f](x) = \sum_{q=1}^N \mathbf{1}_{A_q}(x) \langle \tilde{\eta}_q, f \rangle_{H^s}.
\end{equation}

\begin{theorem}
\label{thm:op-norm-main-tight}
Under Assumption \ref{asm:manifold_dim}, assume $d_q \ge d_{\min} > 0$ for all $q \in V^N$, and $s \in (p/2, \min\{2, p/2 + 1\})$. For any confidence level $\delta \in (0,1)$, the operator norm satisfies
\begin{equation}
    \|\mathcal{D}_N^\omega\|_{H^s \to L^2} \le C_S  \left( \frac{\sqrt{2\ln(2/\delta)}}{\sqrt{d_{\min}}} + \frac{4\ln(2/\delta)}{3\sqrt{N}} \right)
\end{equation}
with probability at least $1-\delta$.
\end{theorem}

\begin{proof}
The operator decomposes as $\mathcal{D}_N^\omega = \sum_{q=1}^N X_q$, where each $X_q = \mathbf{1}_{A_q} \otimes \tilde{\eta}_q \in \mathcal{H}_{HS}$ and $\{X_q\}$ are independent with $\mathbb{E}_z[X_q] = 0$.

Since $\|X_q\|_{HS} = \|\mathbf{1}_{A_q}\|_{L^2} \|\tilde{\eta}_q\|_{H^s}$ and $\|\mathbf{1}_{A_q}\|_{L^2} = \sqrt{1/N}$, by Lemma \ref{lem:eta_bounded_tight}, we define the almost sure bound
\begin{equation}
    \|X_q\|_{HS} \le 2C_S \sqrt{\frac{1}{N}} \triangleq  M_N.
\end{equation}

The variance parameter is bounded by
\begin{equation}
    \sum_{q=1}^N \mathbb{E}_z \left[ \|X_q\|_{HS}^2 \right] = \sum_{q=1}^N \frac{1}{N} \mathbb{E}_z \left[ \|\tilde{\eta}_q\|_{H^s}^2 \right] \le \sum_{q=1}^N \frac{1}{N} \frac{C_S^2}{d_q}.
\end{equation}
Using $d_q \ge d_{\min}$, the above inequality reduces to $\sigma^2 = C_S^2  / d_{\min}$.

Applying Lemma \ref{thm:hilbert-bernstein} and using $\|\mathcal{D}_N^\omega\|_{H^s \to L^2} \le \|\mathcal{D}_N^\omega\|_{HS}$ yields
\begin{equation}
    \mathbb{P}_z \left( \|\mathcal{D}_N^\omega\|_{H^s \to L^2} \ge t \right) \le 2 \exp\left( -\frac{t^2/2}{\sigma^2 + M_N t/3} \right).
\end{equation}

Setting the right-hand side to $\delta$ and solving the quadratic equation for $t$ gives
\begin{align*}
    t &= \frac{\ln(2/\delta)}{3} M_N + \sqrt{ \left(\frac{\ln(2/\delta)}{3} M_N\right)^2 + 2\ln(2/\delta) \sigma^2 }\\
    &\le \frac{2\ln(2/\delta)}{3} M_N + \sqrt{2\ln(2/\delta)} \sigma.
\end{align*}
Substituting the expressions for $M_N$ and $\sigma$ produces the stated bound.
\end{proof}

\bibliography{references}
\end{document}